\documentclass[a4paper,10pt,article]{amsart}
\usepackage{amsmath,amscd,amssymb,amsthm,verbatim,indentfirst,dsfont,mathrsfs,ipa,graphicx,textcomp,enumerate}
\usepackage[all]{xy}
\usepackage{ipa,booktabs,longtable,xypic}
\usepackage[numbers]{natbib}
\usepackage{commutative-diagrams}
\usepackage[colorlinks=true,
linkcolor=blue,
anchorcolor=blue,
citecolor=blue]{hyperref}

\newtheorem{theorem}{Theorem}[section]
\newtheorem{lemma}[theorem]{Lemma}
\newtheorem{corollary}[theorem]{Corollary}
\newtheorem{proposition}[theorem]{Proposition}
\newtheorem{conjecture}[theorem]{Conjecture}

\theoremstyle{definition}
\newtheorem{definition}[theorem]{Definition}

\theoremstyle{remark}
\newtheorem{remark}[theorem]{Remark}

\newtheorem{notation}[theorem]{Notation}
\newtheorem*{proof-thm-prod-coefficients}{Proof of Theorem \ref{prod-coefficients}}
\numberwithin{equation}{section}

\newcommand{\N}{\mathbb{N}}
\newcommand{\R}{\mathbb{R}}
\newcommand{\C}{\mathbb{C}}
\newcommand{\K}{\mathfrak{K}}
\newcommand{\ep}{\varepsilon}
\newcommand{\XA}{\ell^2(Z_X)\otimes H \otimes H_A}
\newcommand{\YA}{\ell^2(Z_Y)\otimes H \otimes H_A}
\newcommand{\XH}{\ell^2(Z_X)\otimes H}
\newcommand{\YH}{\ell^2(Z_Y)\otimes H}
\newcommand{\U}{\mathcal{U}}
\DeclareMathOperator{\supp}{supp}
\DeclareMathOperator{\prop}{prop}

\title{On the quantitative coarse Baum-Connes conjecture with coefficients}

\begin{document}

\author{Jianguo Zhang}
\address{School of Mathematics and Statistics, Shaanxi Normal University}
\email{jgzhang@snnu.edu.cn}

\thanks{The author was supported by NSFC12171156, 12271165, 12301154.}
\date{\today}
	
\begin{abstract}
	In this paper, we introduce the quantitative coarse Baum-Connes conjecture with coefficients (or QCBC, for short) for proper metric spaces which refines the coarse Baum-Connes conjecture. And we prove that QCBC is derived by the coarse Baum-Connes conjecture with coefficients which provides many examples satisfying QCBC. In the end, we show QCBC can be reduced to the uniformly quantitative coarse Baum-Connes conjecture with coefficients of a sequence of bounded metric spaces.
\end{abstract}
\pagestyle{plain}
\maketitle

	
\section{Introduction}
For a proper metric space $X$ with a locally finite net $N_X$, the \textit{Roe algebra} of $X$ is defined to be the norm closure of the $\ast$-algebra consisting of all $(N_X)\times (N_X)$-matrices $T$ with coefficients in $\K(H)$ satisfying that $T_{x,y}=0$ when $d(x, y)>R$ for some $R>0$, where $\K(H)$ is the algebra of compact operators on a Hilbert space $H$ (see also Definition \ref{Def-Roe-algebra}). For an elliptic differential operator $D$ on a Riemannian manifold $M$, Roe defined a higher index for $D$ in the $K$-theory of the Roe algebra of $M$ which generalizes the Fredholm index (cf. \cite{Roe1988}\cite{Roe1993}). The coarse Baum-Connes conjecture furnish a formula to compute the higher indexes, namely, the conjecture assert that the assembly map from the coarse $K$-homology of $M$ to the $K$-theory of the Roe algebra of $M$ is an isomorphism (cf. \cite{HigsonRoe-CBC}\cite{Yu-CBC}, see also Conjecture \ref{CBC}). Moreover, the conjecture also play a key role in the Novikov conjecture and the existence of metrics with positive scalar curvature on a manifold \cite{HigsonRoe-Book}\cite{WillettYu-Book}. The coarse Baum-Connes conjecture has been confirmed for many metric spaces (cf. \cite{Yu1998}\cite{Yu2000}\cite{HigsonRoe-CBC}\cite{Fukaya-Oguni-2020}\cite{Deng-Wang-Yu-2023}) and also disproved for some examples (cf. \cite{Higson-Lafforgue-Skandalis}\cite{Khukhro-Li-Vigolo-Zhang}).

The quantitative $K$-theory is a refinement of the $K$-theory and was first introduced by Yu in \cite{Yu1998} for Roe algebras in order to consider the coarse Baum-Connes conjecture for metric spaces with finite asymptotic dimension. Then Oyono-Oyono and Yu extended the quantitative $K$-theory for all filtered $C^{\ast}$-algebras (cf. \cite{OyonoYu2015}, see also Definition \ref{Def-quan-K}) and applied it to the K$\ddot{\text{u}}$nneth formula for $C^{\ast}$-algebras \cite{OyonoYu2019}. Later on, Chung generalized the quantitative $K$-theory for filtered Banach algebras and used it to consider the $\ell^p$ Baum-Connes conjecture (cf. \cite{Chung-quan-K}\cite{Chung2021}). The quantitative $K$-theory is also applied to the $K$-theory of groupoid $C^{\ast}$-algebras (cf. \cite{Clement-2018}\cite{Oyono-Oyono-2023}), the $\ell^p$ (coarse) Baum-Connes conjecture (cf. \cite{Zhang-Zhou-2021}\cite{Wang-Wang-Zhang-Zhou-2022}), the universal coefficient theorem for $KK$-theory (\cite{Willett-Yu-2024}) and the scalar curvature on a manifold (cf. \cite{Guo-Yu-2023}\cite{Wang-Xie-Yu-2024}). 

In \cite{OyonoYu2015}, Oyono-Oyono and Yu formulated the quantitative Baum-Connes conjecture to compute the quantitative $K$-theory of group $C^{\ast}$-algebras and proved that this conjecture is equivalent to the Baum-Connes conjecture with coefficients. Inspired by this and motivated by computing the quantitative $K$-theory of Roe algebras, in this paper, we introduce the quantitative coarse Baum-Connes conjecture with coefficients (see Conjecture \ref{quan-CBC}) by using of the $K$-theory of localization algebras with coefficients. And we show that the quantitative coarse Baum-Connes conjecture with coefficients implies the original conjecture, namely, we have the following theorem (see also Theorem \ref{quan-CBC-CBC}).

\begin{theorem}\label{relation-1}
	Let $A$ be a $C^{\ast}$-algebra, the quantitative coarse Baum-Connes conjecture with coefficients in $A$ implies the coarse Baum-Connes conjecture with coefficients in $A$.
\end{theorem}   

In order to explore the reverse direction of the above theorem, we first introduce the following \textit{uniform coarse Baum-Connes conjecture with coefficients} (see also Definition \ref{Def-uniform-CBC}):

\begin{definition}
	Call $X$ satisfies the \textit{uniform coarse Baum-Connes conjecture with coefficients} in $A$, if the coarse Baum-Connes conjecture with coefficients in $A$ holds for $\sqcup_{\N} X$, where $\sqcup_{\N} X$ is the separated coarse union of $X$.
\end{definition}

Then by computing the $K$-theory for localization algebras with coefficients of a sequence of metric spaces, we show the following main theorem of this paper (see also Theorem \ref{uniformCBC=quanCBC}):

\begin{theorem}\label{main-thm}
	Let $X$ be a proper metric space and $A$ be a $C^{\ast}$-algebra. Then the followings are equivalent:
	\begin{enumerate}
		\item $X$ satisfies the uniform coarse Baum-Connes conjecture with coefficients in $A$;
		\item $X$ satisfies the quantitative coarse Baum-Connes conjecture with coefficients in $A$. 
	\end{enumerate}
\end{theorem}

In \cite[Corollary 4.12]{Zhang-CBCFC}, we proved that the uniform coarse Baum-Connes conjecture with coefficients in $A$ is derived by the coarse Baum-Connes conjecture with coefficients in $\ell^{\infty}(\N, A\otimes \K(H))$, thus we have the following corollary (see also Corollary \ref{Cor-CBCFC-QCBC}) which provides many examples satisfying the quantitative coarse Baum-Connes conjecture with coefficients.

\begin{corollary}\label{main-Cor}
	Let $X$ be a proper metric space and $A$ be a $C^{\ast}$-algebra. If $X$ satisfies the coarse Baum-Connes conjecture with coefficients in $\ell^{\infty}(\N, A\otimes \K(H))$, then $X$ satisfies the quantitative coarse Baum-Connes conjecture with coefficients in $A$.
\end{corollary}

In the end, we simplify the quantitative coarse Baum-Connes conjecture to a uniform version of the conjecture for a sequence of bounded subspaces. For a proper metric space $X$, fixed a point $x_0\in X$, let $X_n=\{x\in X: n^3-n\leq d(x, x_0) \leq (n+1)^3+(n+1)\}$ which is a bounded subspace of $X$ for every $n\geq 0$. And let $Z=\sqcup_{n\: odd} X_n$. Then we have the following theorem (see also Theorem \ref{redution-bounded-case}). 

\begin{theorem}\label{redu-theorem}
	If the uniformly quantitative coarse Baum-Connes conjecture with coefficients in $A$ holds for $(X_n)_{n\:even}$, $(X_n)_{n\:odd}$ and $(X_n\cap Z)_{n\:even}$, then the quantitative coarse Baum-Connes conjecture with coefficients in $A$ holds for $X$.
\end{theorem}

The paper is organized as follows. In Section \ref{revist-quan-K}, we recall the quantitative $K$-theory and its properties. In Section \ref{Sec-quanCBC}, we introduce the Roe algebras with coefficients, localization algebras with coefficients, discuss their quantitative $K$-theory and formulate the quantitative coarse Baum-Connes conjecture with coefficients. In Section \ref{Sec-Relations}, we study the $K$-theory for localization algebras with coefficients for a sequence of metric spaces and prove Theorem \ref{main-thm}. In Section \ref{Sec-Reduction}, we reduce the quantitative coarse Baum-Connes conjecture with coefficients to a sequence of bounded metric spaces.

\section{A revisit to the quantitative $K$-theory}\label{revist-quan-K}
In this section, we will recall some concepts and properties for the quantitative $K$-theory based on \cite{OyonoYu2015}.
\subsection{Basic concepts}\label{Cept-quan-K}
\begin{definition}(\cite[Definition 1.1]{OyonoYu2015})\label{Def-filtration}
	Let $A$ be a $C^{\ast}$-algebra. A \textit{filtration} of $A$ is a family of closed linear subspaces $(A_r)_{r>0}$ such that
	\begin{itemize}
		\item $A_r\subseteq A_{r'}$, if $r\leq r'$;
		\item $A_r$ is closed by involution;
		\item $A_rA_{r'}\subseteq A_{r+r'}$;
		\item $\cup_{r>0}A_r$ is dense in $A$. 
	\end{itemize}
	If $A$ is unital, we also require the unit of $A$ belongs to $A_r$ for any $r>0$. A $C^{\ast}$-algebra equipped with a filtration is called a \textit{filtered $C^{\ast}$-algebra}.
\end{definition} 

\begin{definition}(\cite{OyonoYu2015})\label{Def-quasielement}
	Let $A$ be a unital filtered $C^{\ast}$-algebra with the unit $I$ and let $0<\ep<1/4$, $r>0$.
	\begin{enumerate}
		\item An element $p$ in $A$ is called an \textit{$(\ep, r)$-projection}, if $p\in A_r$, $p=p^{\ast}$ and $\|p^2-p\|<\ep$. Denote $P^{\ep, r}(A)$ be the set of all $(\ep, r)$-projections in $A$.
		\item An element $u$ in $A$ is called an \textit{$(\ep, r)$-unitary}, if $u\in A_r$ and $\max\{\|uu^{\ast}-I\|, \|u^{\ast}u-I\|\}<\ep$. Denote $U^{\ep, r}(A)$ be the set of all $(\ep, r)$-unitaries in $A$.
	\end{enumerate}
\end{definition}

\begin{remark}\label{Rmek-almost-proj-proj}
	\begin{enumerate}
		\item For an $(\ep, r)$-projection $p$ in a filtered $C^{\ast}$-algebra $A$, then the spectrum of $p$ is contained in $(\frac{1-\sqrt{1+4\ep}}{2}, \frac{1-\sqrt{1-4\ep}}{2})\cup (\frac{1+\sqrt{1-4\ep}}{2}, \frac{1+\sqrt{1+4\ep}}{2})$. Let $\chi_0$ be a real-valued continuous function on $\R$ which satisfies $\chi_0(x)=0$ if $x\leq \frac{1-\sqrt{1-4\ep}}{2}$ and $\chi_0(x)=1$ if $x\geq \frac{1+\sqrt{1-4\ep}}{2}$. Then by the continuous functional calculus, $\chi_0(p)$ is a projection in $A$ and $\|\chi_0(p)-p\|<2\ep$. 
		\item Any $(\ep, r)$-unitary $u$ in $A$ is invertible, and set $\chi_1(u)=u(u^{\ast}u)^{-1/2}$, then $\chi_1(u)$ is a unitary such that $\|\chi_1(u)-u\|<\ep$.
	\end{enumerate}
\end{remark}

For a positive integer $n$, the matrix algebra $M_n(A)$ of a unital filtered $C^{\ast}$-algebra $A$ is a filtered $C^{\ast}$-algebra equipped with a filtration $(M_n(A_r))_{r>0}$. For any $0<\ep<1/4$ and $r>0$, we denote $P^{\ep, r}_n(A)$ and $U^{\ep, r}_n(A)$ to be $P^{\ep, r}(M_n(A))$ and $U^{\ep, r}(M_n(A))$, respectively. Now we consider the following inclusions:
$$P^{\ep, r}_n(A)\hookrightarrow P^{\ep, r}_{n+1}(A)\:\:\textrm{by}\:\:p \mapsto \begin{pmatrix}
	p & 0\\
	0 & 0
\end{pmatrix}.$$
and 
$$U^{\ep, r}_n(A)\hookrightarrow U^{\ep, r}_{n+1}(A)\:\:\textrm{by}\:\:u \mapsto \begin{pmatrix}
	u & 0\\
	0 & I
\end{pmatrix}.$$
Taking the direct limit under the above inclusions, we define
$$P^{\ep, r}_{\infty}(A):=\lim_{n\rightarrow\infty}P^{\ep, r}_n(A);$$
$$U^{\ep, r}_{\infty}(A):=\lim_{n\rightarrow\infty}U^{\ep, r}_n(A).$$\par

For a filtered $C^{\ast}$-algebra $A$, let $C([0,1],A)$ be the $C^{\ast}$-algebra of all continuous functions from $[0,1]$ to $A$. Then $C([0,1],A)$ is a filtered $C^{\ast}$-algebra equipped with a filtration $(C([0,1],A_r))_{r>0}$.

\begin{definition}(\cite[Definition 1.5]{OyonoYu2015})\label{Def-homotopy}
	Let $A$ be a unital filtered $C^{\ast}$-algebra and $0<\ep<1/4$, $r>0$.
	\begin{enumerate}
		\item Two elements $p$ and $p'$ in $P^{\ep, r}_{\infty}(A)$ are called to be \textit{$(\ep, r)$-homotopic}, if there exists an element $p$ in $P^{\ep, r}_{\infty}(C([0,1],A))$ such that $p_0=p$ and $p_1=p'$, where $p=(p_t)_{t\in[0,1]}$ is called an \textit{$(\ep, r)$-homotopy} connecting $p$ and $p'$.
		\item Two elements $u$ and $u'$ in $U^{\ep, r}_{\infty}(A)$ are called to be \textit{$(\ep, r)$-homotopic}, if there exists an element $u$ in $U^{\ep, r}_{\infty}(C([0,1],A))$ such that $u_0=u$ and $u_1=u'$, where $u=(u_t)_{t\in[0,1]}$ is called an \textit{$(\ep, r)$-homotopy} connecting $u$ and $u'$.    
	\end{enumerate}
\end{definition}

\begin{lemma}\label{orth-homotopy}
	Let $A$ be a unital filtered $C^{\ast}$-algebra and $0<\ep<1/4$, $r>0$. Then for any $(\ep, r)$-projection $p$ in $M_n(A)$, we have $p\oplus (I_n-p)$ is $(\ep, r)$-homotopic to $I_n\oplus 0$.
\end{lemma}

\begin{proof}
	For $t\in [0,1]$, let 
	$$p_t=
	\begin{pmatrix}
		p & 0\\
		0 & 0
	\end{pmatrix}
	+
	\begin{pmatrix}
		(\cos\frac{\pi t}{2})^2(I_n-p)                & \cos\frac{\pi t}{2}\sin\frac{\pi t}{2}(I_n-p) \\
		\cos\frac{\pi t}{2}\sin\frac{\pi t}{2}(I_n-p) & (\sin\frac{\pi t}{2})^2(I_n-p) 
	\end{pmatrix},$$
	then $(p_t)_{t\in [0,1]}$ is an $(\ep, r)$-homotopy connecting $I_n\oplus 0$ and $p\oplus (I_n-p)$ in $P^{\ep, r}_{2n}(A)$.
\end{proof}

\begin{lemma}\label{uni-homotopy}
	Let $A$ be a unital filtered $C^{\ast}$-algebra and $0<\ep<1/4$, $r>0$. Then $u\oplus u^{\ast}$ is $(3\ep, 2r)$-homotopic to $I_{2n}$ for any $u\in U^{\ep, r}_{n}(A)$.
\end{lemma}

\begin{proof}
	Let $$u_t=
	\begin{cases}
		\begin{pmatrix}
			I_n+2t(u^{\ast}u-I) & 0\\
			0           & I_n
		\end{pmatrix}          &\quad\text{for $t\in[0,\frac{1}{2}]$},\\
		\begin{pmatrix}
			\sin\pi t & -\cos\pi t\\
			\cos\pi t & \sin\pi t
		\end{pmatrix}
		\begin{pmatrix}
			u^{\ast} & 0\\
			0 & I_n
		\end{pmatrix}
		\begin{pmatrix}
			\sin\pi t & \cos\pi t\\
			-\cos\pi t & \sin\pi t
		\end{pmatrix}
		\begin{pmatrix}
			u & 0\\
			0 & I_n
		\end{pmatrix}             &\quad\text{for $t\in[\frac{1}{2}, 1]$},\\
	\end{cases}$$
	then $(u_t)_{t\in[0,1]}$ is a $(3\ep, 2r)$-homotopy connecting $I_{2n}$ and $u\oplus u^{\ast}$.
\end{proof}

For a non-unital filtered $C^{\ast}$-algebra $A$, set $A^{+}=A\times \C$ equipped with the multiplication $(a,x)\cdot(b,y)=(ab+xb+ya,xy)$ for any $a,b\in A$ and $x,y\in \C$, and the norm on $A^{+}$ is defined to be $\|(a,x)\|=\sup_{b\in A}\|ab+xb\|/\|b\|$, then $A^+$ is a filtered $C^{\ast}$-algebra equipped with a filtration $(A_r\times \C)_{r>0}$. Define $\pi: A^{+}\rightarrow \C$ by $(a,x)\mapsto x$, then $\pi$ is a $\ast$-homomorphism which is contractible, and $\pi$ can be extended to a $\ast$-homomorphism from $M_n(A^{+})$ to $M_n(\C)$ for any positive integer $n$. Define
$$\Tilde{A}=
\begin{cases}
	A  &\quad\text{if $A$ unital,}\\
	A^{+} &\quad\text{if $A$ non-unital.}\\
\end{cases}$$

Now, we give a definition of the quantitative $K$-theory for filtered $C^{\ast}$-algebras.

\begin{definition}\label{Def-quan-K}
	Let $A$ be a filtered $C^{\ast}$-algebra and $0<\ep<1/4$, $r>0$.
	\begin{enumerate}
		\item If $A$ is unital, define
		$$K^{\ep, r}_0(A)=\{(p)_{\ep, r}-(q)_{\ep, r}: p, q\in P^{\ep, r}_{\infty}(A)\}/\sim;$$
		and if $A$ is non-unital, define
		$$K^{\ep, r}_0(A)=\{(p)_{\ep, r}-(q)_{\ep, r}: p, q\in P^{\ep, r}_{\infty}(A^{+}),\:\:\textrm{dim}\:\chi_0(\pi(p))=\textrm{dim}\:\chi_0(\pi(q))\}/\sim,$$
		where $(p)_{\ep,r}-(q)_{\ep, r}\sim (p')_{\ep,r}-(q')_{\ep, r}$ if and only if there exists a positive integer $k$ such that $p\oplus q'\oplus I_k$ is $(\ep, r)$-homotopic to $p'\oplus q\oplus I_k$. Denote by $[p]_{\ep, r}-[q]_{\ep, r}$ the equivalence class of $(p)_{\ep, r}-(q)_{\ep, r}$ modulo $\sim$. $[p]_{\ep, r}$ and $-[q]_{\ep, r}$ are the abbreviations of $[p]_{\ep, r}-[0]_{\ep, r}$ and $[0]_{\ep, r}-[q]_{\ep, r}$, respectively.
		\item Define
		$$K^{\ep, r}_1(A)=U^{\ep, r}_{\infty}(\Tilde{A})/\sim;$$
		where two $(\ep, r)$-unitaries $u\sim u'$ if and only if $u$ is $(3\ep, 2r)$-homotopic to $u'$. Denote by $[u]_{\ep, r}$ the equivalence class of $u$ modulo $\sim$.
	\end{enumerate}
\end{definition}

\begin{remark}
	For a non-unital filtered $C^{\ast}$-algebra $A$, we note that $\textrm{dim}\:\chi_0(\pi(p))=\textrm{dim}\:\chi_0(\pi(p'))$ if $p$ is $(\ep, r)$-homotopic to $p'$ for any two elements $p, p'\in P^{\ep, r}_{\infty}(A^{+})$. 
\end{remark}

\begin{lemma}\label{quan-K-abel}
	Let $A$ be a filtered $C^{\ast}$-algebra. Then $K^{\ep, r}_0(A)$ is an abelian group equipped with the addition operation defined by $([p]_{\ep, r}-[q]_{\ep, r})+([p']_{\ep, r}-[q']_{\ep, r})=[p\oplus p']_{\ep, r}-[q\oplus q']_{\ep, r}$ and with the zero element $[0]_{\ep, r}$. Furthermore, $K^{\ep, r}_1(A)$ is an abelian group equipped with the addition operation defined by $[u]_{\ep, r}+[u']_{\ep, r}=[u\oplus u']_{\ep, r}$ and with the zero element $[I]_{\ep, r}$.
\end{lemma}

\begin{proof}
	We give a proof for non-unital case and it is similar for unital case. The key point is to prove two addition operations defined in the lemma are commutative. For any two elements $p, p'$ in $P^{\ep, r}_{\infty}(A^{+})$, then $p\oplus p'$ is $(\ep, r)$-homotopic to $p'\oplus p$ through the following path for $t\in[0,1]$
	$$\begin{pmatrix}
		\cos\frac{\pi t}{2} & \sin\frac{\pi t}{2}\\
		-\sin\frac{\pi t}{2} & \cos\frac{\pi t}{2}
	\end{pmatrix}
	\begin{pmatrix}
		p & 0\\
		0 & p'
	\end{pmatrix}
	\begin{pmatrix}
		\cos\frac{\pi t}{2} & -\sin\frac{\pi t}{2}\\
		\sin\frac{\pi t}{2} & \cos\frac{\pi t}{2}
	\end{pmatrix},
	$$ 
	which implies the addition operation in $K^{\ep, r}_0(A)$ is commutative. Besides, the inverse of $[p]_{\ep, r}-[q]_{\ep, r}$ is $[q]_{\ep, r}-[p]_{\ep, r}$, thus $K^{\ep, r}_0(A)$ is an abelian group. \par
	Similar arguments as above show that the addition operation in $K^{\ep, r}_1(A)$ is commutative. And for any element $u\in U^{\ep, r}_{\infty}(A^{+})$, the inverse of $[u]_{\ep, r}$ is $[u^{\ast}]_{\ep, r}$ by Lemma \ref{uni-homotopy}. Thus $K^{\ep, r}_1(A)$ is an abelian group.
\end{proof}

\begin{remark}
	For any $0<\ep\leq\ep'<1/4$ and $0<r\leq r'$, we have the following natural group homomorphisms for $\ast=0, 1$:
	$$\iota^{\ep, \ep', r, r'}_{\ast}: K^{\ep, r}_{\ast}(A)\rightarrow K^{\ep', r'}_{\ast}(A), \:\: [p]_{\ep, r}-[q]_{\ep, r}\mapsto [p]_{\ep', r'}-[q]_{\ep', r'}.$$
\end{remark}

In \cite{OyonoYu2015}, Oyono-Oyono and Yu gave a definition for the quantitative $K$-theory which is the same for $K^{\ep, r}_1(A)$ but not for $K^{\ep, r}_0(A)$. Next, we will show our definition for $K^{\ep, r}_0(A)$ is isomorphic to Oyono-Oyono and Yu's definition.


\begin{definition}(\cite[Definition 1.13]{OyonoYu2015})\label{OY-Quan-K}
	Let $A$ be a filtered $C^{\ast}$-algebra and $0<\ep<1/4$, $r>0$. Define
	$$\Bar{K}^{\ep,r}_0(A)=\{(p, l)_{\ep, r}\in P^{\ep, r}_{\infty}(A)\times \N\}/\sim$$
	for $A$ unital and 
	$$\Bar{K}^{\ep,r}_0(A)=\{(p, l)_{\ep, r}\in P^{\ep, r}_{\infty}(A^+)\times\N: \:\: \dim\chi_0(\pi(p))=l\}/\sim$$
	for $A$ non-unital,
	where $(p, l)\sim (p', l')$ if and only if there exists a positive integer $k$ such that $p\oplus I_{l'+k}$ is $(\ep, r)$-homotopic to $p'\oplus I_{l+k}$. Denoted by $[p, l]_{\ep, r}$ the equivalence class of $(p, l)_{\ep, r}$ modulo $\sim$.
\end{definition}

\begin{remark}
	$\Bar{K}^{\ep, r}_0(A)$ is an abelian group equipped with the addition operation defined by $[p,l]_{\ep, r}+[p',l']_{\ep, r}=[p\oplus p', l+l']_{\ep, r}$ and if $p$ is an $(\ep, r)$-projection in $M_n(\Tilde{A})$ for some positive integer $n$, then the inverse of $[p,l]_{\ep, r}$ is $[I_n-p, n-l]_{\ep, r}$ for any $l\leq n$.
\end{remark}

Now, we define a homomorphism between abelian groups:
$$\rho^{\ep, r}: \Bar{K}^{\ep, r}_0(A)\rightarrow K^{\ep, r}_0(A),\:\: [p,l]_{\ep, r}\mapsto [p]_{\ep, r}-[I_l]_{\ep, r}.$$

\begin{lemma}
	Let $A$ be a filtered $C^{\ast}$-algebra. For any $0<\ep<1/4$, $r>0$, the above homomorphism $\rho^{\ep, r}$ is an isomorphism from $\Bar{K}^{\ep, r}_0(A)$ to $K^{\ep, r}_0(A)$.
\end{lemma}

\begin{proof}
	Firstly, $\rho^{\ep, r}$ is well-defined. And if $\rho^{\ep, r}([p,l]_{\ep, r})=[p]_{\ep, r}-[I_l]_{\ep, r}=0$ in $K^{\ep, r}_0(A)$, then there exists an positive integer $k$ such that $p\oplus I_k$ is $(\ep, r)$-homotopic to $I_{l+k}$ which implies that $[p,l]_{\ep, r}=0$ in $\Bar{K}^{\ep, r}_0(A)$. Thus, $\rho^{\ep, r}$ is injective. Next, we will show $\rho^{\ep, r}$ is surjective. For any element $[p]_{\ep, r}-[q]_{\ep, r}$ in $K^{\ep, r}_0(A)$, assume $p$ and $q$ are in $P^{\ep, r}_n(\Tilde{A})$ for some positive integer $n$. Then $[p\oplus (I_n-q), n]_{\ep, r}$ is in $\Bar{K}^{\ep, r}_0(A)$ and $\rho^{\ep, r}([p\oplus (I_n-q), n]_{\ep, r})=[p\oplus (I_n-q)]_{\ep, r}-[I_n]_{\ep, r}$. By Lemma \ref{orth-homotopy}, we know that $(I_n-q)\oplus q$ is $(\ep, r)$-homotopic to $I_n$, which implies that $[p\oplus (I_n-q)]_{\ep, r}-[I_n]_{\ep, r}=[p]_{\ep, r}-[q]_{\ep, r}$ in $K^{\ep, r}_0(A)$. Thus, we completed the proof. 
\end{proof}

The quantitative $K$-theory is a refinement of the $K$-theory. On the other hand, the $K$-theory can be approached by quantitative $K$-theory as $r\rightarrow \infty$. More precisely, for a filtered $C^{\ast}$-algebra $A$ and $0<\ep<1/4$, $r>0$, define (see Remark \ref{Rmek-almost-proj-proj})
$$\iota^{\ep, r}_0: K^{\ep, r}_0(A)\rightarrow K_0(A)\:\: [p]_{\ep, r}-[q]_{\ep, r}\mapsto [\chi_0(p)]-[\chi_0(q)],$$
and 
$$\iota^{\ep, r}_1: K^{\ep, r}_1(A)\rightarrow K_1(A)\:\: [u]_{\ep, r}\mapsto [\chi_1(u)].$$                                                                               
Then we have the following lemma.

\begin{lemma}\label{quan-K-and-K}
	Let $A$ be a filtered $C^{\ast}$-algebra equipped with a filtration $(A_r)_{r>0}$.
	\begin{enumerate}
		\item For any $0<\ep<1/4$ and any $y\in K_{\ast}(A)$, there exists a positive number $r$ and $x\in K^{\ep, r}_{\ast}(A)$ such that $\iota^{\ep, r}_{\ast}(x)=y$.
		\item For any $0<\ep<1/64$ and $r>0$. If an element $x\in K^{\ep, r}_{\ast}(A)$ satisfies $\iota^{\ep, r}_{\ast}(x)=0$, then there exists $r'\geq r$ such that $\iota^{\ep, 16\ep, r, r'}_{\ast}(x)=0$ in $K^{16\ep, r'}_{\ast}(A)$.
	\end{enumerate}
\end{lemma}

\begin{proof}
	(1) For $\ast=0$, let $y=[p]-[q]$ for some projections $p, q\in M_n(\Tilde{A})$. Since that $\cup_{r>0}M_n(\Tilde{A}_r)$ is dense in $M_n(\Tilde{A})$, thus there exists $r>0$ and two self-adjoint elements $p', q'\in M_n(\Tilde{A}_r)$ such that $\max\{\|p'-p\|,\|q'-q\|\}<\ep/4$ which implies that $p', q'$ are two $(\ep, r)$-projections. Take $x=[p']_{\ep, r}-[q']_{\ep, r}$ which is in $K^{\ep, r}_0(A)$ and $\iota^{\ep, r}_0(x)=[\chi_0(p')]-[\chi_0(q')]$. For any $t\in [0,1]$, let $p_t=tp'+(1-t)p$, then $\|p_t-p\|<\ep/4<1/4$. By the continuous functional calculus, $(\chi_0(p_t))_{t\in[0,1]}$ is a homotopy of projections connecting $\chi_0(p')$ and $\chi_0(p)=p$. Thus $[\chi_0(p')]=[p]$. Similarly, we have $[\chi_0(q')]=[q]$. Consequently, $\iota^{\ep, r}_0(x)=y$. \par
	For $\ast=1$, let $y=[u]$ for some unitary $u\in M_n(\Tilde(A))$. Then there exists $r>0$ and an element $u'\in M_n(\Tilde{A}_r)$ such that $\|u'-u\|<\ep/3$ which implies that $u'$ is an $(\ep, r)$-unitary. Take $x=[u']_{\ep, r}$ which is in $K^{\ep, r}_1(A)$ and $\iota^{\ep, r}_1(x)=[\chi_1(u')]$. Let $u_t=tu'+(1-t)u$ for $t\in[0,1]$, then $\|u_t-u\|<\ep/3<1$ which implies that $u_t$ is invertible. Thus $(\chi_1(u_t))_{t\in[0,1]}$ is a homotopy of unitaries connecting $\chi_1(u')$ and $u$. Consequently, $\iota^{\ep, r}_1(x)=y$. \par
	(2) For $\ast=0$, it is sufficient to prove that if $\chi_0(p)$ is homotopic to $0$ for an $(\ep, r)$-projection $p\in M_n(\Tilde{A})$, then there exists $r'>r$ such that $p$ is $(16\ep, r')$-homotopic to $0$. Let $(p_t)_{t\in[0,1]}$ is a homotopy of projections connecting $\chi_0(p)$ and $0$. Then we can choose $p_1, p_2, \cdots, p_m$ on $(p_t)_{t\in[0,1]}$ such that $p_1=\chi_0(p)$, $p_m=0$ and $\|p_i-p_{i+1}\|<\ep/6$ for $i=1, 2, \cdots, m-1$. Since that $\cup_{r>0}M_n(\Tilde{A}_r)$ is dense in $M_n(\Tilde{A})$, thus there exists $r'>r$ and a sequence of self-adjoint elements $p'_1, p'_2, \cdots, p'_m$ in $M_n(\Tilde{A}_{r'})$ such that $\|p'_i-p_i\|<\ep/16$ for $i=1, 2, \cdots, m$. Take $p'_0=p$, connecting all linear homotopies from $p'_i$ to $p'_{i+1}$ for $i=0, 1, \cdots, m-1$, we obtain a $(16\ep, r')$-homotopy connecting $p$ and $0$.\par
	For $\ast=1$, the proof is similar as above.
\end{proof}

\begin{remark}
	For $0<\ep \leq \ep'<1/4$ and $0<r \leq r'$, we have the following commutative diagram:
	$$\xymatrix{
		K^{\ep, r}_{\ast}(A) \ar[d]_{\iota^{\ep, r}} \ar[r]^{\iota^{\ep, \ep', r, r'}} & K^{\ep', r'}_{\ast}(A) \ar[dl]^{\iota^{\ep', r'}} \\
		K_{\ast}(A) &
	}$$
\end{remark}

In the last of this subsection, we will discuss something about Lipschitz homotopy.

\begin{definition}\label{Def-Lip}
	Let $B$ be a $C^{\ast}$-algebra and $c$ be a positive number, a map $h:[0,1]\rightarrow A$ is called \textit{c-Lipschitz}, if $\|h(t)-h(s)\|\leq c|t-s|$ for any $t, s\in [0,1]$.
\end{definition}

The following lemma is coming from \cite[Proposition 1.30]{OyonoYu2015}. For the convenience of readers, we migrate the proof from \cite{OyonoYu2015} to here.

\begin{lemma}\label{Lip-Homotopy}
	Let $A$ be a unital filtered $C^{\ast}$-algebra and $0<\ep<1/4$, $r>0$.
	\begin{enumerate}
		\item If $p$ and $p'$ are $(\ep, r)$-homotopic in $P^{\ep, r}_n(A)$, then there exist two integers $k$, $l$ and a $2$-Lipschitz $(\ep, r)$-homotopy connecting $p\oplus I_k\oplus 0_l$ and $p'\oplus I_k\oplus 0_l$ in $P^{\ep, r}_{n+k+l}(A)$, where $k$ and $l$ depend on the homotopy connecting $p$ and $p'$.
		\item If $u$ and $u'$ are $(\ep, r)$-homotopic in $U^{\ep, r}_n(A)$, then there exists an integer $k'$ and a $4$-Lipschitz $(3\ep, 2r)$-homotopy connecting $u\oplus I_{k'}$ and $u'\oplus I_{k'}$ in $U^{3\ep, 2r}_{n+k}(A)$, where $k$ depends on the homotopy connecting $u$ and $u'$. 
	\end{enumerate} 
\end{lemma}

\begin{proof}
	(1) Firstly, let us recall two facts. The first one is that $I\oplus 0$ is $(\ep, r)$-homotopic to $0\oplus I$ through the following $2$-Lipschitz homotopy:
	$$\begin{pmatrix}
		\cos\frac{\pi t}{2} & \sin\frac{\pi t}{2}\\
		-\sin\frac{\pi t}{2} & \cos\frac{\pi t}{2}
	\end{pmatrix}
	\begin{pmatrix}
		I & 0\\
		0 & 0
	\end{pmatrix}
	\begin{pmatrix}
		\cos\frac{\pi t}{2} & -\sin\frac{\pi t}{2}\\
		\sin\frac{\pi t}{2} & \cos\frac{\pi t}{2}
	\end{pmatrix}.
	$$
	The other one is that $I_n\oplus 0$ is $(\ep, r)$-homotopic to $p\oplus (I_n-p)$ through the following $2$-Lipschitz homotopy:
	$$\begin{pmatrix}
		p & 0\\
		0 & 0
	\end{pmatrix}
	+
	\begin{pmatrix}
		(\cos\frac{\pi t}{2})^2(I_n-p)                & \cos\frac{\pi t}{2}\sin\frac{\pi t}{2}(I_n-p) \\
		\cos\frac{\pi t}{2}\sin\frac{\pi t}{2}(I_n-p) & (\sin\frac{\pi t}{2})^2(I_n-p) 
	\end{pmatrix}.
	$$
	Let $(p_t)_{t\in[0,1]}$ be an $(\ep, r)$-homotopy connecting $p$ and $p'$. Then we choose a sequence of points $p_0, p_1,\cdots, p_k$ such that $p_0=p$, $p_k=p'$ and $\|p_i-p_{i-1}\|<\inf_{t\in[0,1]}(\ep-\|p^2_t-p_t\|)/4$ for $i=1,\cdots, k$. Thus $p_i$ is $(\ep, r)$-homotopic to $p_{i-1}$ through the linear homotopy which is $1$-Lipschitz. Therefore, $p_0\oplus I_{nk}\oplus 0$ is $2$-Lipschitz $(\ep, r)$-homotopic to $p_0\oplus I_n\oplus0\oplus \cdots \oplus I_n\oplus 0$ which is $2$-Lipschitz $(\ep, r)$-homotopic to $p_0\oplus (I_n-p_1)\oplus p_1\oplus \cdots \oplus (I_n-p_k)\oplus p_k$ which is $1$-Lipschitz $(\ep, r)$-homotopic to $p_0\oplus (I_n-p_0)\oplus p_1\oplus \cdots \oplus p_{k-1}\oplus (I_n-p_{k-1})\oplus p_k$ which is $2$-Lipschitz $(\ep, r)$-homotopic to $I_n\oplus 0\oplus\cdots\oplus I_n\oplus 0\oplus p_k$ which is $2$-Lipschitz $(\ep, r)$-homotopic to $p_k\oplus I_{nk}\oplus 0$. As a result, $p\oplus I_{nk}\oplus 0$ is $(\ep, r)$-homotopic to $p'\oplus I_{nk}\oplus 0$ by a $2$-Lipschitz homotopy in $P^{\ep, r}_{n+2kn}(A)$.\par
	(2) Let $(u_t)_{t\in[0,1]}$ be a homotopy connecting $u$ and $u'$, then we can choose a sequence of points $u_0,u_1,\cdots, u_k$ such that $u_0=u$, $u_k=u'$ and $\|u_i-u_{i-1}\|<\inf_{t\in[0,1]}(\ep-\|u^{\ast}_t u_t-I_n\|)/3$ for $i=1,\cdots, k$. Thus $u_i$ is $(\ep, r)$-homotopic to $u_{i-1}$ through the linear homotopy which is $1$-Lipschitz. Firstly, $u_0\oplus I_{nk}$ is $(3\ep, 2r)$-homotopic to $u_0\oplus u^{\ast}_1u_1 \oplus\cdots \oplus u^{\ast}_k u_k$ by linear homotopy which is $1$-Lipschitz. Moreover, $u_0\oplus u^{\ast}_1u_1 \oplus\cdots \oplus u^{\ast}_k u_k\oplus I_{nk}=(I_n\oplus u^{\ast}_1\oplus\cdots\oplus u^{\ast}_k\oplus I_{nk})(u_0\oplus u_1\oplus\cdots\oplus u_k\oplus I_{nk})$ is $4$-Lipschitz $(3\ep, 2r)$-homotopic to $(u^{\ast}_1\oplus\cdots\oplus u^{\ast}_k\oplus I_{nk+n})(u_0\oplus u_1\oplus\cdots\oplus u_k\oplus I_{nk})=u^{\ast}_1u_0 \oplus\cdots \oplus u^{\ast}_k u_{k-1}\oplus u_k\oplus I_{nk}$ which is $(3\ep, 2r)$-homotopic to $u^{\ast}_0u_0 \oplus\cdots \oplus u^{\ast}_{k-1} u_{k-1}\oplus u_k\oplus I_{nk}$ through the linear $1$-Lipschitz homotopy which is $(3\ep, 2r)$-homotopic to $I_n \oplus\cdots \oplus I_n\oplus u_k\oplus I_{nk}$ through the linear $1$-Lipschitz homotopy again which is $2$-Lipschitz $(\ep, r)$-homotopic to $u_k\oplus I_{2nk}$. Consequently, $u\oplus I_{2nk}$ is $(3\ep, 2r)$-homotopic to $u'\oplus I_{2nk}$ by a $4$-Lipschitz homotopy in $U^{3\ep, 2r}_{n+2nk}(A)$.
\end{proof}

\subsection{Properties of the quantitative $K$-theory}\label{Prop-quan-K}
In this subsection, we shall discuss some properties of the quantitative $K$-theory inspired by the similar phenomena in the $K$-theory.

Firstly, we consider the functorial property of the quantitative $K$-theory.

\begin{definition}\label{fil-homo}
	Let $A$, $B$ be two filtered $C^{\ast}$-algebras and $\phi: A\rightarrow B$ be a $\ast$-homomorphism. $\phi$ is called a \textit{controlled homomorphism}, if there exists a non-decreasing function $c_{\phi}: [0, \infty)\rightarrow [0, \infty)$ such that $\phi(A_r)\subseteq B_{c_{\phi}(r)}$. Where $c_{\phi}$ is called a \textit{control function} for $\phi$. Furthermore, if $\phi(A_r)\subseteq B_{r}$ for any $r>0$, then $\phi$ is called a \textit{filtered homomorphism}.
\end{definition}

For any $0<\ep<1/4$, $r>0$ and $\ast=0, 1$, a controlled homomorphism $\phi: A\rightarrow B$ induces a family of homomorphisms $\phi^{\ep, r}_{\ast}: K^{\ep, r}_{\ast}(A)\rightarrow K^{\ep, c_{\phi}(r)}_{\ast}(B)$ and $\phi_{\ast}: K_{\ast}(A) \rightarrow K_{\ast}(B)$. Moreover, we have the following commutative diagram:
$$\xymatrix{
	K^{\ep, r}_{\ast}(A) \ar[d]_{\iota^{\ep, r}_{\ast}} \ar[r]^{\phi^{\ep, r}_{\ast}} & K^{\ep, c_{\phi}(r)}_{\ast}(B) \ar[d]^{\iota^{\ep, c_{\phi}(r)}_{\ast}} \\
	K_{\ast}(A) \ar[r]^{\phi_{\ast}} & K_{\ast}(B).
}$$

\begin{definition}\label{control-multiplier}
	Let $A$ be a filtered $C^{\ast}$-algebra equipped with a filtration $(A_r)_{r>0}$ and $\mathcal{M}(A)$ be the multiplier algebra of $A$. An element $w$ in $\mathcal{M}(A)$ is called to be \textit{controlled}, if there exists a non-decreasing function $c_w: [0,\infty)\rightarrow [0,\infty)$ such that $wa$ and $aw$ are in $A_{c_{w}(r)}$ for any $a\in A_r$. Where $c_w$ is called a \textit{control function} for $w$. In particular, if $wa, aw\in A_{r}$ for any $a\in A_r$ and $r>0$, then $w$ is called to be \textit{filtered}.
\end{definition}

\begin{remark}
	A controlled isometry $w\in \mathcal{M}(A)$ induces a controlled homomorphism $\phi_w: A\rightarrow A$ defined by $\phi_w(a)=waw^{\ast}$.
\end{remark}

\begin{lemma}\label{quan-key-lemma}
	Let $\phi: A\rightarrow B$ be a controlled homomorphism with a control function $c_{\phi}$ and $w$ be a controlled isometry in $\mathcal{M}(B)$ with a control function $c_w$. Then 
	$$\phi^{\ep, r'}_{w,\ast}\circ \phi^{\ep, r}_{\ast}=\iota^{\ep, \ep, r', r''}\circ \phi^{\ep, r}_{\ast}$$ 
	as homomorphisms from $K^{\ep, r}_{\ast}(A)$ to $K^{\ep, r''}_{\ast}(B)$ for any $0<\ep<1/4$, $r>0$ and $\ast=0, 1$, where $r'=c_{\phi}(r)$ and $r''=c^8_w(r')$.
\end{lemma}

\begin{proof}
	Firstly, if $w$ is a controlled unitary, then for any $a\in A_r$, we have that $w\phi(a) w^{\ast}\oplus 0=(w\oplus w^{\ast})(\phi(a)\oplus 0)(w^{\ast}\oplus w)$. By the similar proof of Lemma \ref{uni-homotopy}, $w\oplus w^{\ast}$ is homotopic to $I\oplus I$. Thus $(w\phi(a) w^{\ast})\oplus 0$ is $(\ep, c^4_w(r'))$-homotopic to $\phi(a)\oplus 0$.\par
	Secondly, for general controlled partial isometry $w$. Take
	$$u=\begin{pmatrix}
		w & I-ww^{\ast}\\
		0 & w^{\ast}
	\end{pmatrix},$$
	then $u$ is a controlled unitary in $\mathcal{M}(M_2(B))$ with $c_u=c^2_w$. Thus $w\phi(a) w^{\ast}\oplus 0=u(\phi(a))u^{\ast}$ which is $(\ep, c^4_u(r'))$-homotopic to $\phi(a)\oplus 0$. Consequently, $\phi^{\ep, r'}_{w,\ast}\circ \phi^{\ep, r}_{\ast}=\iota^{\ep, \ep, r', r''}\circ \phi^{\ep, r}_{\ast}$, where $r'=c_{\phi}(r)$ and $r''=c^8_w(r')$.
\end{proof}

\begin{corollary}\label{isom-quan-K}
	If $\phi: A\rightarrow B$ is a filtered homomorphism and $w\in \mathcal{M}(B)$ is a filtered isometry. Then 
	$$\phi^{\ep, r}_{w,\ast}\circ \phi^{\ep, r}_{\ast}=\phi^{\ep, r}_{\ast}$$ 
	for any $0<\ep<1/4$, $r>0$ and $\ast=0, 1$.
\end{corollary}

Next, we refine the notion of quasi-stable (cf. Definition \ref{Def-quasi-stable}) for a family of filtered $C^{\ast}$-algebras.
\begin{definition}\label{uni-quasi-stable}
	A family of filtered $C^{\ast}$-algebras $(A_i)_{i\in \mathcal{I}}$ equipped with filtrations $(A_{i,r})_{i\in \mathcal{I}, r>0}$ is called \textit{uniformly quasi-stable}, if there exists a function $c:[0,\infty)\rightarrow [0,\infty)$ such that for any positive integer $n$, there exists a family of isometries $(w_i)_{i\in \mathcal{I}}\in \mathcal{M}(M_n(A_i))_{i\in \mathcal{I}}$ satisfying that
	\begin{enumerate}
		\item $w_iw^{\ast}_i=e_{11}=I\oplus 0_{n-1}$ for any $i\in \mathcal{I}$;
		\item $(w_ia_i)_{i\in \mathcal{I}}$ and $(a_iw_i)_{i\in \mathcal{I}}$ are in $M_n(A_{i,c(r)})_{i\in \mathcal{I}}$ for any element $(a_i)_{i\in \mathcal{I}}\in M_n(A_{i,r})_{i\in \mathcal{I}}$ and $r>0$.
	\end{enumerate}
	Here $c$ is called a \textit{uniformly control function} for $(A_i)_{i\in \mathcal{I}}$.
\end{definition}

\begin{definition}\label{uniformproduct}\cite[Definition 5.14]{OyonoYu2019}
	Let $(A_i)_{i\in \mathcal{I}}$ be a family of filtered $C^{\ast}$-algebras equipped with filtrations $(A_{i,r})_{i\in \mathcal{I}, r>0}$. The \textit{uniform product} of $(A_i)_{i\in \mathcal{I}}$, denoted by $\prod^{uf}_{i\in \mathcal{I}} A_i$, is defined to be the norm closure of $\cup_{r>0}(\prod_{i\in \mathcal{I}} A_{i, r})$ in $\prod_{i\in \mathcal{I}} A_i$.
\end{definition}

\begin{remark}
	The uniform product $\prod^{uf}_{i\in \mathcal{I}} A_i$ is a filtered $C^{\ast}$-algebra equipped with a filtration $(\prod_{i\in \mathcal{I}} A_{i, r})_{r>0}$.
\end{remark}

\begin{lemma}\label{quan-K-prod}
	Let $(A_i)_{i\in \mathcal{I}}$ be a family of uniformly quasi-stable filtered $C^{\ast}$-algebras, the natural filtered homomorphisms 
	$$\pi_j:\prod^{uf}_{i\in \mathcal{I}} A_i\rightarrow A_j, \:\: (a_i)_{i\in \mathcal{I}}\mapsto a_j,$$
	induce homomorphisms between abelian groups
	$$\Pi \pi_{i, \ast}: K^{\ep, r}_{\ast}(\prod^{uf}_{i\in \mathcal{I}} A_i)\rightarrow \prod_{i\in \mathcal{I}}K^{\ep, r}_{\ast}(A_i).$$
	Then for any $0<\ep<1/4$, $r>0$ and $\ast=0,1$, the following statements hold:
	\begin{enumerate}
		\item \label{quan-K-prod-1} if $\Pi \pi_{i, \ast}(x)=0$, then $\iota^{\ep, \ep, r, c^2(r)}_{\ast}(x)=0$ in $K^{\ep, c^2(r)}_{\ast}(\prod^{uf}_{i\in \mathcal{I}} A_i)$;
		\item \label{quan-K-prod-2} for any $y\in \prod_{i\in \mathcal{I}}K^{\ep, r}_{\ast}(A_i)$, there exists an element $x\in K^{\ep, c^2(r)}_{\ast}(\prod^{uf}_{i\in \mathcal{I}} A_i)$ such that $\iota^{\ep, \ep, c^2(r), c^{10}(r)}_{\ast}\circ\Pi \pi_{i, \ast}(x)=\iota^{\ep, \ep, r, c^{10}(r)}_{\ast}(y)$,
	\end{enumerate}
	where $c$ is a uniformly control function appeared in Definition \ref{uni-quasi-stable}.
\end{lemma}

\begin{proof}
	Let $(w_i)_{i\in \mathcal{I}}$ be a family of isometries in the definition of uniform quasi-stability (cf. Definition \ref{uni-quasi-stable}).\par
	(1) Assume $\Pi \pi_{i, 0}([(p_i)_{i\in \mathcal{I}}]_{\ep, r}-[(q_i)_{i\in \mathcal{I}}]_{\ep, r})=([p_i]_{\ep, r}-[q_i]_{\ep, r})_{i\in \mathcal{I}}=0$, where $(p_i)_{i\in \mathcal{I}}, (q_i)_{i\in \mathcal{I}}\in P^{\ep, r}_{n}(\prod^{uf}_{i\in \mathcal{I}} \Tilde{A}_i)$. Then by Lemma \ref{Lip-Homotopy}, there exists a family of positive integers $(k_i)_{i\in \mathcal{I}}$ such that $p_i\oplus I_{k_i}\oplus 0_{k_i}$ is $(\ep, r)$-homotopic to $q_i\oplus I_{k_i}\oplus 0_{k_i}$ by a $2$-Lipschitz homotopy $p_{i, t}$ in $P^{\ep, r}_{n+2k_i}(\Tilde{A}_i)$ for every $i\in \mathcal{I}$. Thus, $(I_n\oplus w_i)p_{i, t}(I_n\oplus w^{\ast}_i)$ is a $2$-Lipschitz $(\ep, c^2(r))$-homotopy connecting $p_i\oplus I\oplus 0_{2k_i-1}$ and $q_i\oplus I\oplus 0_{2k_i-1}$ in $P^{\ep, c^2(r)}_{n+2k_i}(\Tilde{A}_i)$ for every $i\in \mathcal{I}$. Therefore, $(p_i\oplus I)_{i\in \mathcal{I}}$ is $2$-Lipschitz $(\ep, c^2(r))$-homotopic to $(q_i\oplus I)_{i\in \mathcal{I}}$ in $P^{\ep, c^2(r)}_{\infty}(\prod^{uf}_{i\in \mathcal{I}} \Tilde{A}_i)$ since that the homotopy $(I_n\oplus w_i)p_{i, t}(I_n\oplus w^{\ast}_i)$ is in $P^{\ep, c^2(r)}_{n+1}(\Tilde{A}_i)$ for every $i\in \mathcal{I}$. As a result, $[(p_i)_{i\in \mathcal{I}}]_{\ep, c^2(r)}-[(q_i)_{i\in \mathcal{I}}]_{\ep, c^2(r)}=0$ which implies that (\ref{quan-K-prod-1}) holds for $\ast=0$. \par
	Assume $\Pi \pi_{i, 0}([(u_i)_{i\in \mathcal{I}}]_{\ep, r})=([u_i]_{\ep, r})_{i\in \mathcal{I}}=0$, where $(u_i)_{i\in \mathcal{I}}\in U^{\ep, r}_{n}(\prod^{uf}_{i\in \mathcal{I}} \Tilde{A}_i)$. Then by Lemma \ref{Lip-Homotopy}, there exists a family of $(l_i)_{i\in \mathcal{I}}$ such that $u_i\oplus I_{l_i}$ is $(3\ep, 2r)$-homotopic to $I_n\oplus I_{l_i}$ by a $4$-Lipschitz homotopy $u_{i, t}$. Thus $(u_{i, t})_{i\in \mathcal{I}}$ is a $4$-Lipschitz $(3\ep, 2r)$-homotopy that connecting $(u_i\oplus I_{l_i})_{i\in \mathcal{I}}$ and $(I_{n+l_i})_{i\in \mathcal{I}}$ which implies that $[(u_i)_{i\in \mathcal{I}}]_{\ep, r}=0$ in $K^{\ep, r}_{1}(\prod^{uf}_{i\in \mathcal{I}} A_i)$. That means (\ref{quan-K-prod-1}) holds for $\ast=1$.\par
	(2) For any element $([p_i]_{\ep, r}-[q_i]_{\ep, r})_{i\in \mathcal{I}}\in \prod_{i\in \mathcal{I}}K^{\ep, r}_{0}(A_i)$, where $p_i, q_i\in P^{\ep, r}_{n_i}(\Tilde{A}_i)$, since $w_i p_i w^{\ast}_i,\: w_i q_i w^{\ast}_i \in P^{\ep, c^2(r)}_{1}(\Tilde{A}_i)$ for every $i\in \mathcal{I}$, then $[(w_i p_i w^{\ast}_i)_{i\in \mathcal{I}}]_{\ep, c^2(r)}-[(w_i q_i w^{\ast}_i)_{i\in \mathcal{I}}]_{\ep, c^2(r)}$ is in $K^{\ep, c^2(r)}_0(\prod^{uf}_{i\in \mathcal{I}} A_i)$ and we have 
	$$\Pi \pi_{i, 0}([(w_i p_i w^{\ast}_i)_{i\in \mathcal{I}}]_{\ep, c^2(r)}-[(w_i q_i w^{\ast}_i)_{i\in \mathcal{I}}]_{\ep, c^2(r)})=([w_i p_i w^{\ast}_i]_{\ep, c^2(r)}-[w_i q_i w^{\ast}_i]_{\ep, c^2(r)})_{i\in \mathcal{I}}.$$ By Lemma \ref{quan-key-lemma}, $[w_i p_i w^{\ast}_i]_{\ep, c^{10}(r)}-[(w_i q_i w^{\ast}_i)]_{\ep, c^{10}(r)}=[p_i]_{\ep, c^{10}(r)}-[q_i]_{\ep, c^{10}(r)}$ in $K^{\ep, c^{10}(r)}_0(A_i)$ for every $i\in \mathcal{I}$ which implies that \ref{quan-K-prod-2} holds for $\ast=0$.  \par
	For any element $([u_i]_{\ep, r})_{i\in \mathcal{I}}\in \prod_{i\in \mathcal{I}}K^{\ep, r}_{1}(A_i)$. Then $w_i(u_i-I)w^{\ast}_i+I$ is in $U^{\ep, c^2(r)}_{1}(\Tilde{A}_i)$ for every $i\in \mathcal{I}$. Thus $[(w_i(u_i-I)w^{\ast}_i+I)_{i\in \mathcal{I}}]_{\ep, c^2(r)}$ is in $K^{\ep, c^2(r)}_1(\prod^{uf}_{i\in \mathcal{I}} A_i)$ and by Lemma \ref{quan-key-lemma}, $\Pi \pi_{i, 1}([(w_i(u_i-I)w^{\ast}_i+I)_{i\in \mathcal{I}}]_{\ep, c^{10}(r)})=([u_i]_{\ep, c^{10}(r)})_{i\in \mathcal{I}}$ which implies that \ref{quan-K-prod-2} holds for $\ast=1$. 
\end{proof}

\begin{corollary}\label{Cor-quan-K-prod}
	If $(A_i)_{i\in \mathcal{I}}$ is a family of uniformly quasi-stable filtered $C^{\ast}$-algebras with a uniformly control function $c(r)=r$, then
	$$\Pi \pi_{i, \ast}: K^{\ep, r}_{\ast}(\prod^{uf}_{i\in \mathcal{I}} A_i)\rightarrow \prod_{i\in \mathcal{I}}K^{\ep, r}_{\ast}(A_i)$$
	are isomorphic between abelian groups for any $0<\ep<1/4$, $r>0$ and $\ast=0,1$.
\end{corollary}

    Before showing exact sequences of quantitative $K$-theory, let us review a series of notions.

\begin{definition}\cite{OyonoYu2015}\label{Def-control-pair}
	A \textit{control pair} is a pair $(\lambda, h)$, where $\lambda\geq 1$ and $h: (0, 1/(4\lambda))\rightarrow [1,\infty)$ is a map satisfying that there exists a non-decreasing map $g: (0, 1/(4\lambda))\rightarrow [1,\infty)$ such that $h\leq g$.
\end{definition}

\begin{remark}
	The set of all control pair is a partial order set equipped with $(\lambda, h)\leq (\lambda', h')$ if and only if $\lambda\leq \lambda'$ and $h_{\ep}\leq h'_{\ep}$ for all $\ep\in (0, 1/(4\lambda'))$.
\end{remark}

\begin{notation}
	For a filtered $C^{\ast}$-algebra $A$ and $\ast=0,1$, let 
	$$\mathcal{K}_{\ast}(A)=(K^{\ep, r}_{\ast}(A))_{0<\ep<1/4, r>0}.$$
\end{notation}

\begin{definition}\cite[Definition 2.1]{OyonoYu2015}\label{Def-controlled-mor}
	Let $A$, $B$ be two filtered $C^{\ast}$-algebras and $(\lambda, h)$ be a control pair, $i, j\in\{0, 1\}$. A \textit{$(\lambda, h)$-controlled morphism}
	$$\mathcal{F}: \mathcal{K}_i(A)\rightarrow \mathcal{K}_j(B)$$
	is a family of group homomorphisms for $0<\ep<1/(4\lambda)$, $r>0$
	$$F^{\ep, r}: K^{\ep, r}_i(A)\rightarrow K^{\lambda\ep, h_{\ep}r}_j(B)$$
	such that 
	$$F^{\ep', r'}\circ\iota^{\ep, \ep', r, r'}_i=\iota^{\lambda\ep, \lambda\ep', h_{\ep}r, h_{\ep'}r'}_j\circ F^{\ep, r}$$
	for any $0<\ep\leq\ep'<1/(4\lambda)$, $0<r\leq r'$ and $h_{\ep}r\leq h_{\ep'}r'$.
\end{definition}

\begin{definition}\label{Def-contolled-iso}
	A $(\lambda, h)$-controlled morphism $\mathcal{F}$ is called to be a \textit{$(\lambda, h)$-controlled isomorphism}, if there exists a control pair $(\lambda', h')$ such that 
	\begin{itemize}
	\item it $F^{\ep, r}([u]_{\ep, r})=0$, then $[u]_{\lambda'\ep, h'_{\ep}r}=0$ in $K^{\lambda'\ep, h'_{\ep}r}_{i}(A)$;
	\item for any $[u]_{\ep, r}\in K^{\ep, r}_{j}(B)$, there exists $[v]_{\lambda'\ep, h'_{\ep}r}\in K^{\lambda'\ep, h'_{\ep}r}_{i}(A)$ such that $F^{\lambda'\ep, h'_{\ep}r}([v]_{\lambda'\ep, h'_{\ep}r})=[u]_{\lambda\lambda'\ep, h_{\lambda'\ep}h'_{\ep}r}$.  
	\end{itemize}
\end{definition}

\begin{remark}
	By Lemma \ref{quan-K-and-K}, a $(\lambda, h)$-controlled morphism $\mathcal{F}: \mathcal{K}_i(A)\rightarrow \mathcal{K}_j(B)$ induces a group homomorphism
	$$F: K_i(A)\rightarrow K_j(B),$$
	defined by the relations $F\circ\iota^{\ep, r}_i=\iota^{\lambda\ep, h_{\ep}r}_j\circ F^{\ep, r}$ for any $0<\ep<1/(4\lambda)$, $r>0$.  
\end{remark}

\begin{definition}\cite[Definition 2.9]{OyonoYu2015}
	Let $(\lambda, h)$ be a control pair and $i, j, l\in \{0, 1\}$. Let $\mathcal{F}=(F^{\ep, r})_{0<\ep<1/(4\lambda'), r>0}: \mathcal{K}_i(A_1)\rightarrow \mathcal{K}_j(A_2)$ be a $(\lambda', h')$-controlled morphism and $\mathcal{G}=(G^{\ep, r})_{0<\ep<1/(4\lambda''), r>0}: \mathcal{K}_j(A_2)\rightarrow \mathcal{K}_l(A_3)$ be a $(\lambda'', h'')$-controlled morphism. Then the composition
	$$\mathcal{K}_i(A_1)\xrightarrow{\mathcal{F}}\mathcal{K}_j(A_2)\xrightarrow{\mathcal{G}}\mathcal{K}_l(A_3)$$
	is called to be \textit{$(\lambda, h)$-exact} at $\mathcal{K}_j(A_2)$ if
	\begin{itemize}
		\item $\mathcal{G}\circ \mathcal{F}=0$;
		\item for any $0<\ep<1/(4\max\{\lambda \lambda', \lambda''\})$, $r>0$ and $y\in K^{\ep, r}_j(A_2)$ with $G^{\ep, r}(y)=0$, there exists $x\in K^{\lambda\ep, h_{\ep}r}_i(A_1)$ such that
		$$F^{\lambda\ep, h_{\ep}r}(x)=\iota^{\ep, \lambda\lambda'\ep, r, h'_{\lambda\ep}h_{\ep}r}_j(y).$$
	\end{itemize}
	A sequence of controlled morphisms 
	$$\cdots\rightarrow \mathcal{K}_{i_{k-1}}(A_{k-1})\xrightarrow{\mathcal{F}_{k-1}} \mathcal{K}_{i_k}(A_k)\xrightarrow{\mathcal{F}_k} \mathcal{K}_{i_{k+1}}(A_{k+1})\xrightarrow{\mathcal{F}_{k+1}} \mathcal{K}_{i_{k+1}}(A_{k+1})\rightarrow \cdots$$
	is called \textit{$(\lambda, h)$-exact}, if the composition is $(\lambda, h)$-exact at $\mathcal{K}_{i_k}(A_k)$ for every $k$.
\end{definition}

\begin{definition}\cite[Definition 3.1]{OyonoYu2015}
	Let $A$ be a filtered $C^{\ast}$-algebra with a filtration $(A_r)_{r>0}$ and $J$ be an ideal of $A$. Take $J_r=J\cap A_r$. The extension of $C^{\ast}$-algebras
	$$0\rightarrow J \rightarrow A \rightarrow A/J \rightarrow 0$$
	is called a \textit{completely filtered extension}, if the bijective linear map
	$$A_r/J_r\rightarrow (A_r+J)/J$$
	is a complete isometry for any $r>0$, i.e. $\inf_{y\in M_n(J_r)}\|x+y\|=\inf_{y\in M_n(J)}\|x+y\|$ for any $x\in M_n(A_r)$ and any positive integer $n$.
\end{definition}

\begin{remark}
	For a completely filtered extension and any positive integer $n$, $M_n(J)$ and $M_n(A)/M_n(J)$ are also filtered $C^{\ast}$-algebras equipped with filtrations $(M_n(J)\cap M_n(A_r))_{r>0}$ and $(M_n(A_r+J)/M_n(J))_{r>0}$, respectively.
\end{remark}

Oyono-Oyono and Yu give a quantitative six-term exact sequence in \cite{OyonoYu2015}.
\begin{lemma}\cite[Theorem 4.7]{OyonoYu2015}\label{six-term-seq}
	There exists a control pair $(\lambda, h)$ such that for any completely filtered extension 
	$$0 \rightarrow J \xrightarrow{i} A \xrightarrow{q} A/J \rightarrow 0, $$
	the following six-term sequence is $(\lambda, h)$-exact
	$$\xymatrix{
		\mathcal{K}_0(J) \ar[r]^{\mathcal{I}} & \mathcal{K}_0(A) \ar[r]^{\mathcal{Q}} & \mathcal{K}_0(A/J) \ar[d]^{\mathcal{D}_{J, A}} \\
		\mathcal{K}_1(A/J) \ar[u]^{\mathcal{D}_{J, A}} & \mathcal{K}_1(A) \ar[l]^{\mathcal{Q}} & \mathcal{K}_1(J) \ar[l]^{\mathcal{I}} 
	}$$
	where $\mathcal{I}=(i^{\ep, r}_{\ast})_{0<\ep<1/4, r>0}$ as well as $\mathcal{Q}=(q^{\ep, r}_{\ast})_{0<\ep<1/4, r>0}$ are $(1, 1)$-controlled morphisms and $\mathcal{D}_{J, A}=(D^{\ep, r}_{J,K})_{0<\ep<1/4, r>0}$ is the $(\lambda_{\mathcal{D}}, h_{\mathcal{D}})$-controlled boundary map defined in \cite[Section 3.2]{OyonoYu2015} which is natural for controlled homomorphisms between extensions. 
\end{lemma}

Now, we review the controlled Mayer-Vietoris sequence introduced by Oyono-Oyono and Yu in \cite{OyonoYu2019} for the quantitative K-theory.

\begin{definition}\cite[Section 2]{OyonoYu2019}\label{Def-for-MV}
	Let $A$ be a filtered $C^{\ast}$-algebra with a filtration $(A_r)_{r>0}$ and $s$, $c$ be two positive numbers.
	\begin{enumerate}
		\item Let $\Delta_1$ and $\Delta_2$ be two closed linear subspaces of $A_s$. The pair $(\Delta_1, \Delta_2)$ is called a \textit{completely coercive decomposition pair with coercitivity $c$} for $A$ of order $s$, if for any $0<r\leq s$, any positive integer $n$ and $x\in M_n(A_r)$, there exist $x_1\in M_n(\Delta_1\cap A_r)$ and $x_2\in M_n(\Delta_2\cap A_r)$ with $\max\{\|x_1\|, \|x_2\|\}\leq c\|x\|$ such that $x=x_1+x_2$.
		\item Let $\Delta$ be a closed linear subspace of $A_s$. An \textit{$s$-controlled $\Delta$-neighborhood $C^{\ast}$-algebra} is a $C^{\ast}$-subalgebra $A_{\Delta}$ of $A$ satisfying that 
		\begin{itemize}
			\item $A_{\Delta}$ is a filtered $C^{\ast}$-algebra with a filtration $(A_{\Delta}\cap A_r)_{r>0}$;
			\item $\Delta+A_{5s}\cdot\Delta+\Delta\cdot A_{5s}+A_{5s}\cdot\Delta\cdot A_{5s}\subseteq A_{\Delta}$.
		\end{itemize}
		\item Let $S_1$ and $S_2$ be two subsets of $A$. The pair $(S_1, S_2)$ is called to have \textit{complete intersection approximation property with coercitivity $c$}, if for any $\ep>0$, any positive integer $n$ and $x\in M_n(S_1)$, $y\in M_n(S_2)$ with $\|x-y\|<\ep$, there exists $z\in M_n(S_1\cap S_2)$ such that 
		$$\|x-z\|<c\ep \:\:\textrm{and}\:\: \|y-z\|<c\ep.$$
	\end{enumerate}
\end{definition}

\begin{definition}\cite[Definition 2.17]{OyonoYu2019}\label{Def-control-MV-pair}
	Let $A$ be a filtered $C^{\ast}$-algebra with a filtration $(A_r)_{r>0}$ and $s, c>0$. An \textit{$s$-controlled weak Mayer-Vietoris pair with coercitivity $c$} for $A$ is a quadruple $(\Delta_1, \Delta_2, A_{\Delta_1}, A_{\Delta_2})$ such that the followings hold.
	\begin{itemize}
		\item $\Delta_1$ and $\Delta_2$ are two closed linear subspaces of $A_s$ and the pair $(\Delta_1, \Delta_2)$ is a completely coercive decomposition pair with coercitivity $c$ for $A$ of order $s$.
		\item $A_{\Delta_i}$ is an $s$-controlled $\Delta_i$-neighborhood $C^{\ast}$-algebra for $i=1,2$.
		\item The pair $(A_{\Delta_1,r}, A_{\Delta_2,r})$ has the complete intersection approximation property with coercitivity $c$ for any positive number $r$ with $r\leq s$, where $A_{\Delta_i, r}=A_{\Delta_i}\cap A_r$ for $i=1, 2$.
	\end{itemize}
\end{definition}

\begin{lemma}\cite[Theorem 3.10]{OyonoYu2019}\label{Control-MV}
	For every $c>0$, there exists a control pair $(\lambda, h)$ such that for any filtered $C^{\ast}$-algebra $A$, any $s>0$ and any $s$-controlled weak Mayer-Vietoris pair $(\Delta_1, \Delta_2, A_{\Delta_1}, A_{\Delta_2})$ with coercitivity $c$ for $A$, we have a $(\lambda, h)$-exact six-term exact sequence at order $s$, i.e. it is $(\lambda, h)$-exact at $K^{\ep, r}$ for all $r\leq s$:
	$$\xymatrix{
		\mathcal{K}_0(A_{\Delta_1} \cap A_{\Delta_2}) \ar[r]^{(i_0, i'_0)\:\:\:\:\:\:} & \mathcal{K}_0(A_{\Delta_1})\oplus \mathcal{K}_0(A_{\Delta_2}) \ar[r]^{\:\:\:\:\:\:\:\:\:\:\:\:\:\:i_0-i'_0} & \mathcal{K}_0(A) \ar[d]^{\mathcal{D}_{\Delta_1, \Delta_2}} \\
		\mathcal{K}_1(A) \ar[u]_{\mathcal{D}_{\Delta_1, \Delta_2}} & \mathcal{K}_1(A_{\Delta_1})\oplus \mathcal{K}_1(A_{\Delta_2}) \ar[l]^{i_1-i'_1\:\:\:\:\:\:\:\:\:\:\:\:\:\:} & \mathcal{K}_1(A_{\Delta_1} \cap A_{\Delta_2}) \ar[l]^{\:\:\:\:\:\:(i_1, i'_1)} 
	}$$
	where $i_{\ast}$ and $i'_{\ast}$ are induced by inclusion maps for $\ast=0, 1$ and $\mathcal{D}_{\Delta_1, \Delta_2}$ is the controlled boundary map constructed in \cite[Section 3.2]{OyonoYu2019}.
\end{lemma}

\section{the quantitative coarse Baum-Connes conjecture}\label{Sec-quanCBC}

In this section, let $(X,d)$ be a proper metric space (the properness means every closed ball in $X$ is compact). Since that proper metric spaces are separable, thus we can choose a countable dense subset $Z_X$ in $X$. And let $A$ be a $C^{\ast}$-algebra acting on a Hilbert space $H_A$. Fixed $H$ as a separable Hilbert space. 

\subsection{Roe algebras with coefficients}
Roe algebras are introduced by J. Roe in \cite{Roe1993} motivated by index theory on open manifolds \cite{Roe1988}. In this subsection, we introduce a notion of Roe algebra with coefficients in a $C^{\ast}$-algebra.

\begin{definition}\label{Def-propagation}
	Let $X$, $Y$ be two proper metric spaces and $Z_X$, $Z_Y$ be two countable dense subsets of $X$, $Y$, respectively. Let $\chi$ be the characteristic function. 
	\begin{enumerate}
		\item For a bounded operator $T: \XA\rightarrow \YA$, the \textit{support} of $T$, denoted by $\supp(T)$, is defined to be $\{(x, y)\in X\times Y: (\chi_V\otimes I \otimes I) T (\chi_U\otimes I \otimes I) \neq 0\:\:\textrm{for any open neighborhood $U$, $V$ of $x$, $y$, respectively}\}$.
		\item For a bounded operator $T$ on $\XA$, the \textit{propagation} of $T$, denoted by $\prop(T)$, is defined to be $\sup\{d(x, y): (x,y)\in \supp(T)\}$.
		\item A bounded operator $T$ on $\XA$ is called to be \textit{$A$-locally compact}, if for any compact subset $K\subset X$, the operators $\chi_K T$ and $T\chi_K$ are in $\K(\ell^2(Z_X)\otimes H)\otimes A$, where $\K(\ell^2(Z_X)\otimes H)$ is the $C^{\ast}$-algebra of compact operators on $\ell^2(Z_X)\otimes H$.
	\end{enumerate}
\end{definition}

\begin{definition}\label{Def-Roe-algebra}
	The \textit{Roe algebra of $X$ with coefficients in $A$}, denoted by $C^{\ast}(X,A)$, is defined to be the norm closure of the $\ast$-algebra consisting of all $A$-locally compact operators on $\XA$ with finite propagation.
\end{definition}

\begin{remark}
	For different choices of the dense subset $Z_{X}$ of $X$, Roe algebras with coefficients are non-canonically isomorphic and their $K$-theory are canonically isomorphic. 
\end{remark}

\begin{lemma}\label{filtration-Roe-algebra}
	The Roe algebra $C^{\ast}(X, A)$ of $X$ with coefficients in $A$ is a filtered $C^{\ast}$-algebra equipped with the following filtration
	$$C^{\ast}(X, A)_r=\{T\in C^{\ast}(X, A): \prop(T)\leq r\}$$
	for any $r>0$.
\end{lemma}

\begin{proof}
	For any two operators $T, S\in C^{\ast}(X, A)$, we have that $\supp(TS)$ is contained in the closure of the set 
	$$\{(x,y)\in X\times X: \:\: \textrm{there exists $z\in X$ such that $(x,z)\in \supp(S)$, $(z,y)\in \supp(T)$}\}$$
	which implies 
	$$\prop(TS)\leq \prop(T)+\prop(S).$$
	Thus $(C^{\ast}(X, A)_r)_{r>0}$ is a filtration of $C^{\ast}(X, A)$.
\end{proof}

\begin{lemma}\label{Roequasi-stable}
	Let $(X_i)_{i\in \mathcal{I}}$ be a family of proper metric spaces and $(A_i)_{i\in \mathcal{I}}$ be a family of $C^{\ast}$-algebras acting on $H_{A_i}$, respectively. Then a family of filtered $C^{\ast}$-algebras $(C^{\ast}(X_{i}, A_i))_{i\in \mathcal{I}}$ is uniformly quasi-stable with a uniformly control function $c(r)=r$ (see Definition \ref{uni-quasi-stable}).
\end{lemma}

\begin{proof}
	Let $Z_{X_{i}}$ be a countable dense subset of $X_i$ and $H$ be a separable Hilbert space. For any positive integer $n$, define an isometry $V: H^{n} \rightarrow H \oplus H^{n-1}$ with the first coordinate $H$ as its image. Then $I \otimes V \otimes I: \ell^2(Z_{X_{i}}) \otimes H^{n} \otimes H_{A_i} \rightarrow \ell^2(Z_{X_{i}}) \otimes H^{n} \otimes H_{A_i}$ is an isometry in the multiplier algebra of $M_n(C^{\ast}(X_i, A_i))$ with the propagation zero. For any $i\in \mathcal{I}$, we have that $(I \otimes V \otimes I)(I \otimes V \otimes I)^{\ast}=I \otimes VV^{\ast} \otimes I=e_{11}$ and $(I \otimes V \otimes I)T,\: T(I\otimes V \otimes I) \in M_n(C^{\ast}(X_i, A_i)_{r})$ for any element $T \in M_n(C^{\ast}(X_i, A_i)_{r})$ and $r>0$. Thus $(C^{\ast}(X_{i}, A_i))_{i\in \mathcal{I}}$ is uniformly quasi-stable with a uniformly control function $c(r)=r$.
\end{proof}

Combining the above lemma with Corollary \ref{Cor-quan-K-prod}, we have the following formula to compute the quantitative $K$-theory for the uniform product of a family of Roe algebras with coefficients.

\begin{corollary}\label{uniform-product-Roealg}
	Let $(X_i)_{i\in \mathcal{I}}$ and $(A_i)_{i\in \mathcal{I}}$ be as above, then the natural filtered homomorphisms 
	$$\pi_i:\prod^{uf}_{i\in \mathcal{I}} C^{\ast}(X_i, A_i)\rightarrow C^{\ast}(X_i, A_i), \:\: (T_i)_{i\in \mathcal{I}}\mapsto T_i,$$
	induce a isomorphism
	$$\Pi \pi_{i, \ast}: K^{\ep, r}_{\ast}(\prod^{uf}_{i\in \mathcal{I}} C^{\ast}(X_i, A_i)) \rightarrow \prod_{i\in \mathcal{I}}K^{\ep, r}_{\ast}(C^{\ast}(X_i, A_i)),$$
	for any $0<\ep<1/4$ and $r>0$.
\end{corollary}

\begin{definition}\label{Def-coarse-map}
	Let $X$ and $Y$ be two proper metric spaces. A map $f: X\rightarrow Y$ is called a \textit{coarse map}, if 
	\begin{itemize}
		\item the inverse image of $f$ of any compact subset is pre-compact;
		\item for any $R>0$, there exists $S>0$ such that $\sup\{d(f(x),f(y)):\:\: d(x,y)\leq R\}\leq S$.  
	\end{itemize}
	For a coarse map $f$, define a function $w_f: \R^{+}\rightarrow \R^{+}$ by $R \mapsto \sup\{d(f(x),f(y)):\:\: d(x,y)\leq R\}$. 
\end{definition}

\begin{definition}\label{Def-coarse-equi}
	Two proper metric spaces $X$ and $Y$ are called to be \textit{coarsely equivalent}, if there exists two coarse maps $f: X\rightarrow Y$, $g:Y\rightarrow X$ and a constant $c>0$ such that 
	$$\max\{d(gf(x), x), d(fg(y), y)\}\leq c,$$
	for any $x\in X$ and $y\in Y$.
\end{definition}

For a coarse map $f: X\rightarrow Y$ and $\delta>0$. Choose a disjoint Borel cover $(U_i)_{i\in \mathcal{I}}$ of $Y$ with diameter less than $\delta$ and every $U_i$ has non-empty interior, then there exists a family of isometries $V_i: \ell^2(f^{-1}(U_i)\cap Z_X)\otimes H \rightarrow \ell^2(U_i\cap Z_Y)\otimes H$. Let 
$$V_f=\bigoplus_{i\in \mathcal{I}}V_i \otimes I: \XA \rightarrow \YA.$$
Then $V_f$ is an isometry and 
\begin{equation}\label{*}
	\supp(V_f)\subseteq \{(x, y)\in X\times Y: \:\: d(f(x),y)<\delta\}.
\end{equation}
Thus we have the following lemma.

\begin{lemma}\label{covering-isometry}
	For a coarse map $f: X\rightarrow Y$ and $\delta>0$, the above isometry $V_f$ induces a $\ast$-homomorphism
	$$ad_f: C^{\ast}(X, A)\rightarrow C^{\ast}(Y, A),$$
	defined by 
	$$ad_f(T)=V_f T V^{\ast}_f,$$
	that satisfies 
	$$\prop(ad_f(T))\leq w_f(\prop(T))+2\delta.$$
	Moreover, another $\ast$-homomorphism $ad'_f$ defined by another isometry $V'_f$ that satisfies \eqref{*} induces the same homomorphism as $ad_f$ on $K$-theory.     
\end{lemma}

\begin{proof}
	Firstly, if $(y, y')\in \prop(V_f T V^{\ast}_f)$, then there exist $x, x'\in X$ such that $(y, x)\in \supp(V^{\ast}_f)$, $(x, x')\in \supp(T)$ and $(x', y')\in \supp(V_f)$. Thus $d(y, y')\leq d(y, f(x))+d(f(x), f(x'))+d(f(x'), y')\leq w_f(\prop(T))+2\delta$ which implies 
	$$\prop(ad_f(T))\leq w_f(\prop(T))+2\delta.$$ 
	
	Secondly, if $T\in C^{\ast}(X, A)$ is an $A$-locally compact operator. For any compact subset $K$ of $Y$, let $K'=\{x\in X:\:\:\textrm{there exists $y\in K$ such that $(y, x)\in \supp(V^{\ast}_f)$}\}$ that is closed. Since $d(f(x),y)<\delta$ for any $(y, x)\in \supp(V^{\ast}_f)$, thus $K'\subseteq f^{-1}(\overline{B_{\delta}(K)})$ where $\overline{B_{\delta}(K)}=\{y\in Y: d(y,K)\leq \delta\}$ which is compact by the properness of $Y$. Thus $K'$ is a compact subset of $X$ by the properness of $f$. Therefore, we have that $V_f T V^{\ast}_f \chi_K=V_f T \chi_{K'} V^{\ast}_f \chi_K$ is in $\K(\ell^2(Z_X)\otimes H) \otimes A$ because $T$ is $A$-locally compact. Similarly, we can show that $\chi_K V_f T V^{\ast}_f$ is in $\K(\ell^2(Z_X)\otimes H) \otimes A$.
	
	Finally, if there exists another isometry $V'_f: \XA \rightarrow \YA$ satisfying the relation \eqref{*}. Then by similar arguments as above, $V_fV^{\ast}_f$, $V_f'V'^{\ast}_f$, $V'_fV^{\ast}_f$ are all multipliers of $C^{\ast}(Y, A)$. And
	$$\begin{pmatrix}
		0 & 0\\
		0 & V'_f T V'^{\ast}_f 
	\end{pmatrix}
	=
	\begin{pmatrix}
		I-V_fV^{\ast}_f & V_fV'^{\ast}_f\\
		V'_fV^{\ast}_f  & I-V'_fV'^{\ast}_f 
	\end{pmatrix}
	\begin{pmatrix}
		V_f T V^{\ast}_f & 0\\
		0                  & 0 
	\end{pmatrix}
	\begin{pmatrix}
		I-V_fV^{\ast}_f & V_fV'^{\ast}_f\\
		V'_fV^{\ast}_f  & I-V'_fV'^{\ast}_f 
	\end{pmatrix}
	$$
	which implies that $ad'_f$ and $ad_f$ induce the same homomorphism on $K$-theory since that their images are unitarily equivalent by a unitary multiplier. 
\end{proof}

Combining Lemma \ref{filtration-Roe-algebra} and Lemma \ref{quan-key-lemma} with the above lemma, we have the following corollary.
\begin{corollary}\label{quan-cover-isometry}
	Let $f: X\rightarrow Y$ be a coarse map and $\delta>0$, the above $\ast$-homomorphism $ad_f$ is a controlled homomorphism with a control function $c(r)=w_f(r)+2\delta$. Thus $ad_f$ induces a family of homomorphisms 
	$$ad^{\ep, r, c(r)}_{f,\ast}: K^{\ep, r}_{\ast}(C^{\ast}(X, A))\rightarrow K^{\ep, c(r)}_{\ast}(C^{\ast}(Y, A))$$
	for any $0<\ep<1/4$, $r>0$.\\
	Moreover, if another isometry $V'_f$ that satisfies \eqref{*} define a controlled homomorphism $ad'_f$, then
	$$\iota_{\ast}^{\ep, \ep, c(r), c(r)+16\delta}\circ ad'^{\ep, r, c(r)}_{f,\ast}=\iota_{\ast}^{\ep, \ep, c(r), c(r)+16\delta}\circ ad^{\ep, r, c(r)}_{f,\ast}$$
	as homomorphisms from $K^{\ep, r}_{\ast}(C^{\ast}(X, A))$ to $K^{\ep, c(r)+16\delta}_{\ast}(C^{\ast}(Y, A))$ for any $0<\ep<1/4$, $r>0$.  
\end{corollary}

If $f: X\rightarrow Y$ is a coarsely equivalent map, then we can choose $V_f$ as above to be a unitary, thus we have the following corollary.

\begin{corollary}\label{Cor-coarse-equi-quanK}
	If $f:X\rightarrow Y$ is a coarsely equivalent map, then $ad_f$ induces a $(1, c)$-controlled isomorphism from $\mathcal{K}_{\ast}(C^{\ast}(X, A))$ to $\mathcal{K}_{\ast}(C^{\ast}(Y, A))$, where $c$ is a function defined as the above corollary.  
\end{corollary}

\subsection{Localization algebras with coefficients}

\begin{definition}\label{Def-localization-algebra}
	The \textit{localization algebra} of $X$ with coefficients in $A$, denoted by $C^{\ast}_L(X, A)$, is defined to be the norm closure of the $\ast$-algebra consisting of all bounded and uniformly continuous functions $u: [0, \infty)\rightarrow C^{\ast}(X, A)$ satisfying that
	$$\lim_{t\rightarrow \infty}\prop(u(t))=0.$$
	The \textit{propagation} of $u\in C^{\ast}_L(X, A)$, denoted by $\prop(u)$, is defined to be 
	 $$\sup_{t\in [0,\infty)}\prop(u(t)).$$
\end{definition}

Similar to Lemma \ref{filtration-Roe-algebra}, we have the following lemma. 

\begin{lemma}\label{filtration-localization-algebra}
	The localization algebra $C^{\ast}_L(X, A)$ is a filtered $C^{\ast}$-algebra equipped with a filtration
	$$C^{\ast}_L(X, A)_r=\{u\in C^{\ast}_L(X, A): \prop(u)\leq r\}.$$
\end{lemma}

Let $f: X\rightarrow Y$ be a uniformly continuous coarse map and $\{\delta_k\}_{k\in \N}$ be a
deceasing sequence of positive numbers with limit zero. Choose a sequence of Borel covers $\{\U_k\}_{k\in \N}$ of $Y$ such that 
\begin{itemize}
	\item every cover $\U_k$ is disjoint, i.e. $U_{k,i}\cap U_{k,j}=\emptyset$ for different $i, j$;
	\item every $U_{k,i}\in \U_k$ has non-empty interior;
	\item the diameter of $\U_k$ is less than $\delta_k$ for any $k\in \N$;
	\item $\U_{k+1}$ refines $\U_k$, i.e. every set in $\U_{k+1}$ is contained in some set in $\U_k$ for all $k\in \N$.
\end{itemize}

Then $\ell^2(U_{k,i}\cap Z_Y)\otimes H$ is an infinitely dimensional separable Hilbert space for every set $U_{k,i}\in \U_{k}$. Thus we can define an isometry $V_{k,i}: \ell^2(f^{-1}(U_{k,i})\cap Z_X)\otimes H \rightarrow \ell^2(U_{k,i}\cap Z_Y)\otimes H$ with infinitely dimensional cokernel. By the fourth condition as above, there exists $\{U_{k+1,ij}\}_{1\leq j \leq n}\subseteq \U_{k+1}$ such that $\cup^n_{j=1}U_{k+1,ij}=U_{k,i}$. Then we can choose $n$ number of isometries $V_{k+1,ij}: \ell^2(f^{-1}(U_{k+1,ij})\cap Z_X)\otimes H \rightarrow \ell^2(U_{k+1,ij}\cap Z_Y)\otimes H$ with infinitely dimensional cokernel. Let $V_{k+1,i}=\bigoplus^n_{j=1}V_{k+1,ij}: \ell^2(f^{-1}(U_{k,i})\cap Z_X)\otimes H \rightarrow \ell^2(U_{k,i}\cap Z_Y)\otimes H$ that is an isometry with infinitely dimensional cokernel. Let $W_{k, i}$ be a unitary from the cokernel of $V_{k,i}$ to the cokernel of $V_{k+1,i}$ and let $W'_{k,i}=V_{k+1,i}V^{\ast}_{k,i}+0\oplus W$ which is a unitary on $\ell^2(U_{k,i}\cap Z_Y)$. Let $f: S^1\rightarrow [0,1)$ be the Borel inverse function of the continuous function $t\mapsto e^{2\pi i t}$. Define $\gamma_{k,i}$ to be a path of isometries by $\gamma_{k,i}(t)=e^{2\pi itf(W'_{k, i})}V_{k,i}$. Then $\gamma_{k,i}$ is a $2\pi$-Lipschitz path connecting $V_{k, i}$ with $V_{k+1,i}$ and $\supp(\gamma_{k,i}(t))\subseteq f^{-1}(U_{k,i}) \times U_{k,i}$. Define 
$$\gamma_k(t)=\bigoplus_i \gamma_{k,i}(t): \XH \rightarrow \YH$$
for $t\in [0,1]$. Then for $s\in [0, \infty)$, define
$$V_f(s)=\gamma_k(s-k) \otimes I: \ell^2(Z_X)\otimes H \otimes H_A \rightarrow \ell^2(Z_X)\otimes H \otimes H_A,\:\:\textrm{if $s\in [k,k+1]$}.$$
Then $V_f$ is a bounded and uniformly continuous function satisfying that
\begin{equation}\label{**}
	\supp(V_f(s))\subseteq \{(x,y)\in X\times Y: d(f(x), y)<\delta_k\},\:\:\textrm{if $s\in[k,k+1]$}.
\end{equation}
Thus by Lemma \ref{covering-isometry}, we have the following lemma.

\begin{lemma}\label{continuous-covering-isometry}
	Let $f: X\rightarrow Y$ be a uniformly continuous coarse map and $\{\delta_k\}_{k\in \N}$ be a deceasing sequence of positive numbers with limit zero, then $V_f$ as defined above induces a $\ast$-homomorphism 
	$$Ad_f: C^{\ast}_L(X, A) \rightarrow C^{\ast}_L(Y, A)$$
	defined by 
	$$Ad_f(u)(t)=V_f(t)u(t)V^{\ast}_f(t),$$
	that satisfies 
	$$\prop(Ad_f(u)(t))\leq w_f(\prop(u(t)))+2\delta_k\:\:\textrm{for $t\in [k,k+1]$}.$$
	Moreover, if another bounded and uniformly continuous function $V'_f$ that satisfies $V'_f(t)$ is an isometry and \eqref{**}, then $Ad'_f$ defined by $V'_f$ induces the same homomorphism as $Ad_f$ on $K$-theory.
\end{lemma}  

Combining Lemma \ref{filtration-localization-algebra} and Lemma \ref{quan-key-lemma} with the above lemma, we have the following corollary.
\begin{corollary}\label{quan-continuous-cover-isometry}
	Let $f$ and $\{\delta_k\}_{k\in \N}$ be as above. Then the above $\ast$-homomorphism $Ad_f$ is a controlled homomorphism with a control function $c(r)=w_f(r)+2\delta_1$. Thus $Ad_f$ induces a family of homomorphisms 
	$$Ad^{\ep, r, c(r)}_{f,\ast}: K^{\ep, r}_{\ast}(C^{\ast}_L(X, A))\rightarrow K^{\ep, c(r)}_{\ast}(C^{\ast}_L(Y, A))$$
	for any $0<\ep<1/4$, $r>0$.\\
	Moreover, if another bounded and uniformly continuous function $V'_f$ that satisfies $V'_f(t)$ is an isometry and \eqref{**} define a controlled homomorphism $Ad'_f$, then
	$$\iota_{\ast}^{\ep, \ep, c(r), c(r)+16\delta_1}\circ Ad'^{\ep, r, c(r)}_{f, \ast}=\iota_{\ast}^{\ep, \ep, c(r), c(r)+16\delta_1}\circ Ad^{\ep, r, c(r)}_{f, \ast}$$
	as homomorphisms from $K^{\ep, r}_{\ast}(C^{\ast}_L(X, A))$ to $K^{\ep, c(r)+16\delta}_{\ast}(C^{\ast}_L(Y, A))$ for any $0<\ep<1/4$, $r>0$.  
\end{corollary}

Actually, the quantitative $K$-theory of $C^{\ast}_L(X, A)$ is isomorphic to its $K$-theory in the quantitative sense. 

\begin{lemma}\label{quan-K-local-alg}
	For any $0<\ep<1/64$, $r>0$ and $\ast=0,1$, the homomorphism
	$$\iota^{\ep, r}_{\ast}: K^{\ep, r}_{\ast}(C^{\ast}_L(X, A))\rightarrow K_{\ast}(C^{\ast}_L(X, A))$$
	is surjective and if $\iota^{\ep, r}_{\ast}(x)=0$ for $x\in K^{\ep, r}_{\ast}(C^{\ast}_L(X, A))$, then $\iota^{\ep, 16\ep, r, r}_{\ast}(x)=0$ in $K^{16\ep, r}_{\ast}(C^{\ast}_L(X, A))$.
\end{lemma}

\begin{proof}
	Firstly, we show that $\iota^{\ep, r}_{\ast}$ is surjective. For $\ast=0$ and any $[p]-[q]\in K_0(C^{\ast}_L(X, A))$ with $p, q\in M_n(\widetilde{C^{\ast}_L(X, A)})$, then there exist $r'>r$ and $p', q'\in M_n(\widetilde{C^{\ast}_L(X, A)})$ such that $\max\{\|p'-p\|, \|q'-q\|\}<\ep/4$ and $\max\{\prop(p'), \prop(q')\}\leq r'$ which implies $p', q'$ are two $(\ep, r')$-projections and $\iota^{\ep, r'}([p']_{\ep, r'}-[q']_{\ep, r'})=[p]-[q]$. Since that $\lim_{t\rightarrow \infty}\prop(p'(t))=\lim_{t\rightarrow \infty}\prop(q'(t))=0$, thus there exists $t_1>0$ such that $\max\{\prop(p'(t)), \prop(q'(t))\}\leq r$ for all $t\geq t_1$. Let $p''(t)=p'(t+t_1)$ and $q''(t)=q'(t+t_1)$ for $t\geq 0$, then $p'', q''$ are two $(\ep, r)$-projections in $ M_n(\widetilde{C^{\ast}_L(X, A)})$ and $\iota^{\ep, ep, r, r'}_0([p'']_{\ep, r}-[q'']_{\ep, r})=[p']_{\ep, r'}-[q']_{\ep, r'}$. Thus $\iota^{\ep, r}_0([p'']_{\ep, r}-[q'']_{\ep, r})=[p]-[q]$. For $\ast=1$, we can similarly prove.\par
	Secondly, we prove $\iota^{\ep, r}_{0}$ is quantitatively injective. And one can similarly prove for $\iota^{\ep, r}_{1}$. For $[p]_{\ep,r}-[q]_{\ep,r}\in K^{\ep,r}_{0}(C^{\ast}_L(X, A))$ with $\iota^{\ep, r}_{0}([p]_{\ep, r}-[q]_{\ep, r})=[\chi_0(p)]-[\chi_0(q)]=0$, then there exists $n\in \N$ such that $\chi_0(p)\oplus I_n$ is homotopic to $\chi_0(q)\oplus I_n$ by a homotopy $(h_s)_{s\in [0, 1]}$. Choose $h_1, \cdots, h_m$ on $(h_s)_{s\in [0, 1]}$ such that $h_1=\chi_0(p)\oplus I_n$, $h_m=\chi_0(q)\oplus I_n$ and $\|h_i-h_{i+1}\|<\ep/8$ for $i=1, \cdots, m-1$. Then there exist $r'>r$ and a sequence of self-adjoint elements $h'_1, \cdots, h'_m$ such that $\max\{\|h'_i-h_i\|: i=1,\cdots, m\}<\ep/16$ and $\max\{\prop(h'_i): i=1, \cdots, m\}\leq r'$. Take $h'_0=p\oplus I_n$ and $h'_{m+1}=q\oplus I_n$, then $h'_i$ is an $(\ep, r')$-projection for each $i=0, \cdots, m+1$. Since that $\lim_{t\rightarrow \infty}\max\{\prop(h'_i(t)): i=0, \cdots, m+1\}=0$, thus there exists $t_1>0$ such that $\max\{\prop(h'_i(t)): i=0,\cdots, m+1\}\leq r$ for all $t\geq t_1$. Let $h''_i(t)=h'_i(t+t_1)$ for $t\geq 0$ and $i=0, \cdots, m+1$, then every $h''_i$ is an $(\ep, r)$-projection and $h''_0, h''_{m+1}$ are $(\ep, r)$-homotopic to $p\oplus I, q\oplus I_n$, respectively. By linear homotopies between $h''_i$ and $h''_{i+1}$, we obtain that $p\oplus I_n$ is $(16\ep, r)$-homotopic to $q\oplus I_n$ which implies that $\iota^{\ep, 16\ep, r, r}_0([p]_{\ep, r}-[q]_{\ep, r})=0$ in $K^{16\ep, r}_{0}(C^{\ast}_L(X, A))$.      
\end{proof}

\subsection{the quantitative coarse Baum-Connes conjecture with coefficients}

Let $(N, d_N)$ be a locally finite metric space, namely, any ball in $N$ just contains finite elements. For any $k>0$, the \textit{Rips complex} of $N$ at scale $k$, denoted by $P_k(N)$, is a simplicial complex whose set of vertices is $N$ and a subset $\{z_0, \cdots, z_n\}\subseteq N$ spans an $n$-simplex in $P_k(N)$ if and only if $d_N(z_i, z_j)\leq k$ for any $i,j=0, \cdots, n$. 

Firstly, define a metric $d_{S_k}$ on $P_{k}(N)$ to be the path metric whose restriction to each simplex is the standard spherical metric on the unit sphere by mapping $\sum_{i=0}^{n} t_{i}z_{i}$ to 
$$\left(\frac{t_0}{\sqrt{\sum_{i=0}^{n}t^{2}_{i}}}, \cdots, \frac{t_n}{\sqrt{\sum_{i=0}^{n}t^{2}_{i}}}\right),$$
namely, 
$$d_{S_k}\left(\sum_{i=0}^{n} t_{i}z_{i}, \sum_{i=0}^{n} t'_{i}z_{i}\right)=\frac{2}{\pi} \arccos\left(\frac{\sum_{i=0}^{n} t_i t'_i}{\sqrt{(\sum_{i=0}^{n}t^{2}_{i})(\sum_{i=0}^{n}t'^{2}_{i})}}\right),$$
where $t_i, t'_i\in [0,1]$ and $\sum_{i=0}^n t_i=\sum_{i=0}^n t'_i=1$, and the metric $d_{S_k}$ between different connected components is defined to be infinity.
Then, define a metric $d_{P_k}$ on $P_k(N)$ to be
$$d_{P_k}(x, y)=\inf\left\{\sum_{i=0}^{n} d_{S_k}(x_i, y_i)+\sum_{i=0}^{n-1} d_N(y_i, x_{i+1}) \right\},$$
for all $x, y\in P_k(N)$, here the infimum is taken over all sequences of the form $x=x_0, y_0, x_1, y_1,\cdots, x_n, y_n=y$, where $x_1,\cdots, x_n, y_0, \cdots, y_{n-1} \in N$. The advantage of the metric $d_{P_k}$ is that the metric spaces $(N, d_N)$ and $(P_k(N), d_{P_k})$ are coarsely equivalent for any $k\geq 0$ by Proposition 7.2.11 in \cite{WillettYu-Book}.

For any $0<\ep<1/4$ and $r>0$, the evaluation at zero map $e_A$ from $C^{\ast}_L(X, A)$ to $C^{\ast}(X, A)$ induces the following homomorphism:
$$e^{\ep, r}_{A, \ast}: K^{\ep, r}_{\ast}(C^{\ast}_L(X, A))\rightarrow K^{\ep, r}_{\ast}(C^{\ast}(X, A)).$$
For any $x\in K_{\ast}(C^{\ast}_L(X, A))$, there exists $x'\in K^{\ep/16, r}_{\ast}(C^{\ast}_L(X, A))$ such that $\iota^{\ep/16, r}(x')=x$ and $x'$ is unique in $K^{\ep, r}_{\ast}(C^{\ast}_L(X, A))$ by Lemma \ref{quan-K-local-alg}. Thus we obtain a unique element $e^{\ep, r}_{A, \ast}(x')$ in $K^{\ep, r}_{\ast}(C^{\ast}(X, A))$ that satisfies $\iota^{\ep, r}_{\ast}(e^{\ep, r}_{A, \ast}(x'))=e_{A, \ast}(x)$. This gives a well-defined homomorphism:
$$\mu^{\ep, r}_{A,\ast}: K_{\ast}(C^{\ast}_L(X, A))\rightarrow K^{\ep, r}_{\ast}(C^{\ast}(X, A)),\:\: x \mapsto e^{\ep, r}_{A, \ast}(x').$$
and we have $\iota^{\ep, r}_{\ast}(\mu^{\ep, r}_{A, \ast}(x))=e_{A, \ast}(x)$.

By the above definition, we have the following lemma.

\begin{lemma}\label{quan-CBC-and-CBC}
	For any $0<\ep\leq \ep'<1/4$ and $0<r\leq r'$, the following diagrams are commutative
	$$\xymatrix{
		K^{\ep/16, r}_{\ast}(C^{\ast}_L(X,A)) \ar[d]_{\iota^{\ep, r}_{\ast}} \ar[r]^{e^{\ep, r}_{A, \ast}} & K^{\ep, r}_{\ast}(C^{\ast}(X, A)) \ar[d]^{\iota^{\ep, r}_{\ast}}  \\
		K_{\ast}(C^{\ast}_L(X, A)) \ar[ur]^{\mu^{\ep, r}_{A, \ast}} \ar[r]^{e_{A, \ast}} & K_{\ast}(C^{\ast}(X, A)).
	}$$
	and
	$$\xymatrix{
		K_{\ast}(C^{\ast}_L(X, A)) \ar[r]^{\mu^{\ep, r}_{A, \ast}} \ar[dr]_{\mu^{\ep', r'}_{A, \ast}} & K^{\ep, r}_{\ast}(C^{\ast}(X, A)) \ar[d]^{\iota^{\ep, \ep', r, r'}_{\ast}} \\
		&  K^{\ep', r'}_{\ast}(C^{\ast}(X, A)).
	}$$
\end{lemma}

Recall that a \textit{net} in $X$ is a subset $N_{X}$ satisfying that there exists $C\geq 1$ such that
\begin{enumerate}
	\item $d(z, z')\geq C$ for any $z,z'\in N_{X}$;
	\item for any $x\in X$, there exists $z\in N_{X}$ such that $d(z, x)\leq C$.
\end{enumerate}    
By Zorn's lemma, any proper metric space admits a locally finite net.

Let $N_{X}$ be a locally finite net in $X$. As a matter of convenience, we abbreviate the Rips complex $P_{k}(N_X)$ to $P_{k}$ in the rest of this section. For two non-negative integers $k\leq k'$, there is an inclusion map $i_{kk'}: P_{k}\rightarrow P_{k'}$ which is a contractible coarse map. Thus we have the following homomorphisms by Lemma \ref{continuous-covering-isometry} and Lemma \ref{quan-cover-isometry}:
$$Ad_{i_{kk'}, \ast}: K_{\ast}(C^{\ast}_L(P_{k}, A) \rightarrow K_{\ast}(C^{\ast}_L(P_{k'}, A)$$
and 
$$ad^{\ep, r}_{i_{kk'},\ast}: K^{\ep, r}_{\ast}(C^{\ast}(P_{k}, A) \rightarrow K^{\ep, r_{k'}}_{\ast}(C^{\ast}(P_{k'}, A),$$
where $r_{k'}=r+\frac{1}{2^{k'}}$. In order to reflect the change of the propagation, we denote $ad^{\ep, r}_{i_{kk'},\ast}$ by $ad^{\ep, r, r_{k'}}_{i_{kk'},\ast}$ in the rest of the paper.
Consider the following homomorphism:
$$\mu^{\ep, r}_{k,A,\ast}: K_{\ast}(C^{\ast}_L(P_k, A)) \rightarrow K^{\ep, r}_{\ast}(C^{\ast}(P_{k}, A)).$$

\begin{definition}\label{Def-quan-statements}
	Let $0<\ep \leq \ep'<1/4$, $r+1\leq r'$ and $k\leq k'$, we define the following quantitative statements:
	\begin{enumerate}
		\item $QI_{A}(k, k', \ep, r)$: for $x\in K_{\ast}(C^{\ast}_L(P_{k}, A))$, if $\mu^{\ep, r}_{k,A, \ast}(x)=0$, then $Ad_{i_{kk'}, \ast}(x)=0$.
		\item $QS_{A}(k, k', \ep, \ep', r, r')$: for any $y\in K^{\ep, r}_{\ast}(C^{\ast}(P_{k}, A))$, there exists an element $x\in K_{\ast}(C^{\ast}_L(P_{k'}, A))$ such that $\mu^{\ep', r'}_{k',A, \ast}(x)=\iota^{\ep, \ep', r_{k'}, r'}_{\ast} \circ ad^{\ep, r, r_{k'}}_{i_{kk'},\ast}(y)$ in $K^{\ep', r'}_{\ast}(C^{\ast}(P_{k'}, A))$.
	\end{enumerate}
\end{definition}

Now, we introduce the \textit{quantitative coarse Baum-Connes conjecture with coefficients}.

\begin{conjecture}\label{quan-CBC}
	Let $X$ be a proper metric space. Then the followings are true.
	\begin{enumerate}
		\item\label{quan-CBC1} For any $k\in \N$, there exists $0<\ep<1/4$ satisfying that for any $r>0$, there exists $k'\geq k$ such that $QI_A(k, k', \ep, r)$ holds for $X$.
		\item\label{quan-CBC2} For any $k\in \N$, there exists $0<\ep<1/4$ satisfying that for any $r>0$, there exist $k'\geq k$, $r'\geq r+1$ and $1/4>\ep'\geq \ep$ such that $QS_A(k, k', \ep, \ep', r, r')$ holds for $X$.
	\end{enumerate}
\end{conjecture}

\begin{remark}\label{quan-CBC-coa-equi}
	The quantitative coarse Baum-Connes conjecture with coefficients is independent of the choice of locally finite nets in $X$. 
\end{remark}

If $f: X\rightarrow Y$ is a coarsely equivalent map and $N_X$ is a locally finite net in $X$, then $f(N_X)$ is a locally finite net in $Y$, thus by Corollary \ref{Cor-coarse-equi-quanK}, we have the following lemma.

\begin{lemma}\label{Lem-coarse-equi-QCBC}
	Assume that $X$ is coarsely equivalent to $Y$. If $X$ satisfies the quantitative coarse Baum-Connes conjecture with coefficients in $A$, then $Y$ also satisfies the quantitative coarse Baum-Connes conjecture with coefficients in $A$.
\end{lemma}

In the end of this section, we give an equivalent condition for the quantitative coarse Baum-Connes conjecture with coefficients by using the kernel of the evaluation at zero map $e_A$.

 Let
 $$C^{\ast}_{L,0}(X, A)=\{u\in C^{\ast}_L(X, A): u(0)=0\},$$ 
which is a closed ideal of $C^{\ast}_L(X, A)$. And let $C^{\ast}_{L, 0}(X, A)_r=C^{\ast}_{L,0}(X, A) \cap C^{\ast}_{L}(X, A)_r=\{u\in C^{\ast}_{L,0}(X, A): \prop(u)\leq r\}$. Moreover, both $C^{\ast}_{L}(X, A)_r/C^{\ast}_{L, 0}(X, A)_r$ and $(C^{\ast}_{L}(X, A)_r+C^{\ast}_{L, 0}(X, A))/C^{\ast}_{L, 0}(X, A)$ are isometrically isomorphic to $C^{\ast}(X, A)_r$. Thus we obtain a completely filtered extension
$$0\rightarrow C^{\ast}_{L, 0}(X, A)\xrightarrow{j} C^{\ast}_{L}(X, A)\xrightarrow{e} C^{\ast}(X, A)\rightarrow 0.$$
Then by Lemma \ref{six-term-seq}, there exists a control pair $(\lambda, h)$ such that 
\begin{equation}\label{***}
	\xymatrix{
		\mathcal{K}_0(C^{\ast}_{L,0}(P_k, A)) \ar[r]^{\mathcal{J}_{k}} & \mathcal{K}_0(C^{\ast}_{L}(P_k, A)) \ar[r]^{\mathcal{E}_{k}} & \mathcal{K}_0(C^{\ast}(P_k, A)) \ar[d]^{\mathcal{D}_{k}} \\
		\mathcal{K}_1(C^{\ast}(P_k, A)) \ar[u]^{\mathcal{D}_{k}} & \mathcal{K}_1(C^{\ast}_{L}(P_k, A)) \ar[l]^{\mathcal{E}_{k}} & \mathcal{K}_1(C^{\ast}_{L,0}(P_k, A)) \ar[l]^{\mathcal{J}_{k}} 
	}
\end{equation}
is $(\lambda, h)$-exact for all $k\in \N$, where $\mathcal{K}_{\ast}(\cdot)=(K^{\ep, r}_{\ast}(\cdot))_{0<\ep<1/4, r>0}$ and $\mathcal{J}_{k}=(j^{\ep, r}_{k})_{0<\ep<1/4, r>0}$ as well as $\mathcal{E}_{k}=(e^{\ep, r}_{k})_{0<\ep<1/4, r>0}$ are $(1,1)$-controlled morphisms, $\mathcal{D}_{k}=(D^{\ep, r}_{k})_{0<\ep<1/4, r>0}$ is a $(\lambda_{\mathcal{D}},h_{\mathcal{D}})$-controlled morphism which is natural for maps $Ad^{\ep, r, 2r}_{i_{kk'}}$ and $ad^{\ep, r, 2r}_{i_{kk'}}$.

\begin{lemma}\label{equa-quan-CBC}
	Let $\lambda, \lambda_{\mathcal{D}}\geq 1$ be as above. The quantitative coarse Baum-Connes conjecture with coefficients in $A$ holds for $X$ if and only if for any $k\in \N$, there exists $0<\ep_k<1/(4\max\{\lambda, \lambda_{\mathcal{D}}\})$ satisfying that for any $r>0$, there exist $k'\geq k$, $1/(4\max\{\lambda, \lambda_{\mathcal{D}}\})>\ep'\geq \ep_k$ and $r'\geq r+1$ such that 
	$$\iota^{\ep_{k}, \ep', r_{k'}, r'}_{\ast} \circ Ad^{\ep_{k}, r, r_{k'}}_{i_{kk'},\ast}:K^{\ep_{k}, r}_{\ast}(C^{\ast}_{L, 0}(P_k, A))\rightarrow K^{\ep', r'}_{\ast}(C^{\ast}_{L,0}(P_{k'}, A))$$
	is a zero map, where $r_{k'}=r+\frac{1}{2^{k'}}$.
\end{lemma}

\begin{proof}
	Let $(\lambda, h)$ and $(\lambda_{\mathcal{D}}, h_{\mathcal{D}})$ be two control pairs as above. \par
	Firstly, we assume that the quantitative coarse Baum-Connes conjecture with coefficients in $A$ is true for $X$. For $k\in \N$, let $0<\ep<1/(64\lambda\lambda_{\mathcal{D}})$ be an indetermined number. For any $r>0$, consider an element $x\in K^{\ep/16, r}_{\ast}(C^{\ast}_{L, 0}(P_k, A))$. By the above six-term $(\lambda, h)$-exact sequence \eqref{***}, we have that 
	$$e^{\ep, r}_k \circ j^{\ep/16, r}_k(x)=0.$$
	Then by Lemma \ref{quan-CBC-and-CBC}, we have that $\mu^{\ep, r}_{k,\ast} \circ \iota^{\ep, r}_{\ast}(j^{\ep, r}_k(x))=0$. Thus by the first part (\ref{quan-CBC1}) of the quantitative coarse Baum-Connes conjecture with coefficients, there exists $k'\geq k$ such that $Ad_{i_{kk'},\ast} \circ \iota^{\ep, r}_{\ast}(j^{\ep, r}_k(x))=0$ which implies that $\iota^{\ep, r_{k'}}_{\ast} \circ Ad^{\ep, r, r_{k'}}_{i_{kk'},\ast}(j^{\ep, r}_k(x))=0$. Then $\iota^{\ep, 16\ep, r_{k'}, r_{k'}}_{\ast}\circ Ad^{\ep, r, r_{k'}}_{i_{kk'},\ast}(j^{\ep, r}_k(x))=0$ in $K^{16\ep, r_{k'}}_{\ast}(C^{\ast}_L(P_{k'}, A))$ by Lemma \ref{quan-K-local-alg}. Thus $$j^{16\ep, r_{k'}}_{k'}(Ad^{16\ep, r, r_{k'}}_{i_{kk'},\ast}(x))=0,$$ 
	where $Ad^{16\ep, r, r_{k'}}_{i_{kk'},\ast}(x)\in K^{16\ep, r_{k'}}_{\ast}(C^{\ast}_{L,0}(P_{k'}, A))$. Then by six-term $(\lambda, h)$-exact sequence \eqref{***} again, there exists $y\in K^{16\lambda\ep, h_{16\ep}r_{k'}}_{\ast-1}(C^{\ast}(P_{k'}, A))$ such that 
	\begin{equation} \label{eq1}       
		D^{16\lambda\ep,h_{16\ep}r_{k'}}_{k'}(y)=\iota^{16\ep,16\lambda\lambda_{\mathcal{D}}\ep, r_{k'}, h_{\mathcal{D}, 16\lambda\ep}\cdot h_{16\ep}r_{k'}}_{\ast}(Ad^{16\ep, r, r_{k'}}_{i_{kk'},\ast}(x)),
	\end{equation}
	in $K^{16\lambda\lambda_{\mathcal{D}}\ep, h_{\mathcal{D}, 16\lambda\ep}\cdot h_{16\ep}r_{k'}}_{\ast}(C^{\ast}_{L,0}(P_{k'}, A))$. Thus by the second part (\ref{quan-CBC2}) of the quantitative coarse Baum-Connes conjecture, there exist $k''\geq k'$, $r'\geq h_{\mathcal{D}, 16\lambda\ep}\cdot h_{16\ep}r+2$ and $1/4>\ep'\geq 16\lambda\lambda_{\mathcal{D}}\ep>0$ satisfying that there exists $x'\in K_{\ast-1}(C^{\ast}_L(P_{k''}, A))$ such that $\mu^{\ep', r'}_{k'', \ast-1}(x')=\iota^{16\lambda\ep, \ep', h_{16\ep}r_{k'}, r'}\circ ad^{16\lambda\ep, h_{16\ep}r_{k'}, h_{16\ep}r_{k'}}_{i_{k'k''},\ast-1}(y)$ in $K^{\ep', r'}_{\ast-1}(C^{\ast}(P_{k''}, A))$. Since that $\mu^{\ep', r'}_{k'', \ast-1}(x')=e^{\ep', r'}_{k'', \ast-1}(x'')$ for a unique element $x''\in K^{\ep'/16, r'}_{\ast-1}(C^{\ast}_L(P_{k''}, A))$, thus $e^{\ep', r'}_{k'', \ast-1}(x'')=\iota^{16\lambda\ep, \ep', h_{16\ep}r_{k'}, r'}\circ ad^{16\lambda\ep, h_{16\ep}r_{k'}, h_{16\ep}r_{k'}}_{i_{k'k''},\ast-1}(y)$. Then by \eqref{***} again, 
	\begin{equation} \label{eq2}
		D^{\ep', r'}_{k''} \circ e^{\ep', r'}_{k'', \ast-1}(x'')=D^{\ep', r'}_{k''} \circ \iota^{16\lambda\ep, \ep', h_{16\ep}r_{k'}, r'}\circ ad^{16\lambda\ep, h_{16\ep}r_{k'}, h_{16\ep}r_{k'}}_{i_{k'k''},\ast-1}(y)=0.
	\end{equation}   
	Combining \eqref{eq1} with \eqref{eq2} and by the quantitative coarse Baum-Connes conjecture \ref{quan-CBC}, we can choose suitable $0<\ep<1/4$ such that $\iota^{\ep, \ep', r_{k''}, r'}_{\ast} \circ Ad^{\ep, r, r_{k''}}_{i_{kk''},\ast}(x)=0$. 
	
	Secondly, Assume that for any $k\in \N$ there exists $0<\ep_k<1/4$ satisfying that for any $r>0$, there exist $k'\geq k$, $1/4>\ep'\geq \ep_k>0$ and $r'\geq r+1$ such that $\iota^{\ep_k, \ep', r_{k'}, r'}_{\ast}\circ Ad^{\ep_k, r, r_{k'}}_{i_{kk'},\ast}$ is a zero map from $K^{\ep_k, r}_{\ast}(C^{\ast}_{L, 0}(P_k, A))$ to $K^{\ep', r'}_{\ast}(C^{\ast}_{L, 0}(P_{k'}, A))$. Let $0<\ep<\ep_k/\lambda$. For any $r>0$, if $\mu^{\ep, r}_{k, \ast}(x)=0$ for an element $x\in K_{\ast}(C^{\ast}_L(P_k, A))$. Then there exists $x'\in K^{\ep/16, r}_{\ast}(C^{\ast}_L(P_{k}, A))$ with $\iota^{\ep/16, r}(x')=x$ such that 
	$$\mu^{\ep, r}_{k, \ast}(x)=e^{\ep, r}_{\ast}(x')=0.$$
	Thus by the above six-term $(\lambda, h)$-exact sequence \eqref{***}, there exists an element $x''$ in $K^{\lambda\ep, h_{\ep}r}_{\ast}(C^{\ast}_{L, 0}(P_{k}, A))$ such that 
	\begin{equation}\label{eq3}
		j^{\lambda\ep, h_{\ep}r}_{\ast}(x'')=\iota^{\ep, \lambda\ep, r, h_{\ep}r}_{\ast}(x').
	\end{equation}
	By the assumption, there exist $k'\geq k$, $1/4>\ep'\geq \ep_k\geq \lambda\ep$ and $r'\geq h_{\ep}(r+1)$ such that 
	\begin{equation}\label{eq4}
		\iota^{\lambda\ep, \ep', h_{\ep}r_{k'}, r'}_{\ast} \circ Ad^{\lambda\ep, h_{\ep}r, h_{\ep}r_{k'}}_{i_{kk'},\ast}(x'')=0.
	\end{equation}
	Combining \eqref{eq3} \eqref{eq4} with the fact $\iota^{\ep/16, r}(x')=x$, we obtain that $Ad_{i_{kk'},\ast}(x)=0$ which implies that the first part (\ref{quan-CBC1}) of the quantitative coarse Baum-Connes conjecture is true. \par
	Let $k\in \N$ and $0<\ep\leq \ep_k/(\max\{\lambda, \lambda_{\mathcal{D}}\})$. For any $r>0$ and $y\in K^{\ep, r}_{\ast}(C^{\ast}(P_{k}, A))$, consider $D^{\ep, r}_{k}(y)\in K^{\lambda_{\mathcal{D}}, h_{\mathcal{D},\ep}r}_{\ast-1}(C^{\ast}_{L, 0}(P_{k}, A))$. By the assumption, there exist $k'\geq k$, $1/4>\ep'\geq \ep_k \geq \lambda_{\mathcal{D}}\ep$ and $r'\geq h_{\mathcal{D},\ep}(r+1)$ such that 
	$$\iota^{\lambda_{\mathcal{D}}\ep, \ep', h_{\mathcal{D}, \ep}r_{k'}, r'}_{\ast} \circ Ad^{\lambda_{\mathcal{D}}\ep, h_{\mathcal{D}, \ep}r, h_{\mathcal{D}, \ep}r_{k'}}_{i_{kk'}, \ast-1}(D^{\ep, r}_{k}(y))=0,$$
	which implies that 
	$$D^{\ep', r'_{k'}}_{k'}(ad^{\ep', r', r'_{k'}}_{i_{kk'}, \ast}(y))=0.$$
	By the above six-term $(\lambda, h)$-exact sequence, there exists $x\in K^{\lambda\ep', h_{\ep'}r'_{k'}}_{\ast}(C^{\ast}_L(P_{k'}, A))$ such that 
	$$e^{\lambda\ep', h_{\ep'}r'_{k'}}_{k'}=\iota^{\ep', \lambda\ep', r'_{k'}, h_{\ep'}r'_{k'}}_{\ast}(ad^{\ep', r', r'_{k'}}_{i_{kk'}, \ast}(y)).$$
	Then by Lemma \ref{quan-CBC-and-CBC}, we have 
	$$\mu^{\lambda\ep', h_{\ep'}r'_{k'}}_{k', \ast}(\iota^{\lambda\ep', h_{\ep'}r'_{k'}}_{\ast}(x))=\iota^{\ep, \lambda\ep', r, h_{\ep'}r'_{k'}}_{\ast}(ad^{\ep, r, r_{k'}}_{i_{kk'}, \ast}(y)),$$
	which implies that the second part (\ref{quan-CBC2}) of the quantitative coarse Baum-Connes conjecture is true.
\end{proof}

\section{Relations between the quantitative coarse Baum-Connes conjecture with coefficients and the original conjecture}\label{Sec-Relations}

In this section, we will show that the quantitative coarse Baum-Connes conjecture with coefficients is a refinement of the original conjecture and is equivalent to a uniform version of the original conjecture. Firstly, we need consider the left-hand side of the conjecture for a sequence of metric spaces.

\subsection{$K$-theory for localization algebras with coefficients for a sequence of metric spaces} \label{subsection-K-homology-direct-product}

Let $(X_i, d_i)_{i\in \N}$ be a sequence of metric spaces. Define the \textit{separated coarse union} of $(X_i)_{i\in \N}$ to be a metric space $(X=\sqcup X_i, d)$ by
\begin{equation}\label{productmetric}
	d(x,x')=\left \{ 
	\begin{array}{lcl}
		d_i(x,x'), & \mbox{for} & x,x'\in X_i\\
		\infty,    & \mbox{for} & x\in X_i, x'\in X_{i'}, i\neq i'. 
	\end{array}
	\right.
\end{equation}
Then for any $C^{\ast}$-algebra $A$, we have 
$$C^{\ast}(X, A)=\prod^{uf}_{i\in \N}C^{\ast}(X_i, A),$$
where $\prod^{uf}_{i\in \N}C^{\ast}(X_i, A)$ (see also Definition \ref{uniformproduct}) is defined to be the norm closure in $\prod_{i\in N} C^{\ast}(X_i, A)$ of the set
$$\bigcup_{r>0}\left(\prod_{i\in \N} \{T\in C^{\ast}(X_i, A): \prop(T)\leq r\}\right).$$ 
Thus the localization algebra $C^{\ast}_L(X, A)$ is the norm closure of the $\ast$-algebra consisting of all bounded and uniformly continuous functions $u:[0,\infty) \rightarrow \prod^{uf}_{i\in \N}C^{\ast}(X_i, A)$ satisfying that 
$$\lim_{t\rightarrow \infty} \sup_{i\in \N}(\prop(u_i(t)))=0.$$
where $u_i(t)$ is the $i$th coordinate of $u(t)$ in $C^{\ast}(X_i, A)$.

Next, We will show that the $K$-theory for the localization algebra of the separated coarse union of a sequence of metric spaces $(X_i)_{i\in \N}$ can be computed by the $K$-theory for the localization algebra of $X_i$. Firstly, let us recall a series of lemmas in $K$-theory.

\begin{lemma}\label{smallhomotopy}
	Let $B$ be a $C^{\ast}$-algebra and $p, q$ be two projections in $B$. If $\|p-q\|<1$, then there exists a homotopy $(p_t)_{t\in [0,1]}$ of projections in $M_2(B)$ connecting $p\oplus 0$ and $q\oplus 0$ such that $\|p_t-p_{t'}\|\leq 4\pi |t-t'|$ for any $t,t' \in [0,1]$.
\end{lemma}
\begin{proof}
	Firstly, we show that $p$ and $q$ are unitarily equivalent by following the proof of \cite[Proposition 4.1.7]{HigsonRoe-Book}. Let $v=pq+(1-p)(1-q)$, then
	$$v-1=2pq-p-q=(p-q)(2q-1),$$
	thus $\|v-1\|<1$ which implies that $v$ is invertible. Let $u=v(v^{\ast}v)^{-1/2}$, then $u$ is a unitary. Since $pv=vq$ and $v^{\ast}vq=qv^{\ast}v$, thus 
	$$p=uqu^{\ast}.$$ 
	For $t\in [0,1]$, let 
	$$u_t=
	\begin{pmatrix}
		u  &  0\\
		0  &  I
	\end{pmatrix}
	\begin{pmatrix}
		cos\frac{\pi t}{2}   &  sin\frac{\pi t}{2}\\
		-sin\frac{\pi t}{2}  &  cos\frac{\pi t}{2}
	\end{pmatrix}
	\begin{pmatrix}
		u^{\ast}  &  0\\
		0         &  I
	\end{pmatrix}
	\begin{pmatrix}
		cos\frac{\pi t}{2}   &  -sin\frac{\pi t}{2}\\
		sin\frac{\pi t}{2}  &  cos\frac{\pi t}{2}
	\end{pmatrix}.
	$$
	Then $p_t=u_t (q\oplus 0) u^{\ast}_t$ is a homotopy of projections connecting $q\oplus 0$ and $p\oplus 0$. Moreover, since 
	$$\max\{|cos\frac{\pi t}{2}-cos\frac{\pi t'}{2}|, |sin\frac{\pi t}{2}-sin\frac{\pi t'}{2}|\} \leq \frac{\pi}{2}|t-t'|,$$
	thus $\|p_t-p_{t'}\|\leq 4\pi |t-t'|$ for any $t,t'\in [0,1]$.
\end{proof}

\begin{lemma}\cite[Proposition 2.7.5]{WillettYu-Book}\label{isometryequivalent}
	Let $\alpha: C\rightarrow D$ be a $\ast$-homomorphism and $v$ be a partial isometry in the multiplier algebra of $D$ satisfying that $\alpha(c)v^{\ast}v=\alpha(c)$ for all $c\in C$. Define $\beta: C\rightarrow D$ by $\beta(c)=v \alpha(c) v^{\ast}$. Then $\alpha$ and $\beta$ induce the same morphisms on $K$-theory. 
\end{lemma}

Let $\pi_C: C \rightarrow Q$ be a surjective $\ast$-homomorphism and $\pi_D: D \rightarrow Q$ be a $\ast$-homomorphism. Let $P=\{(c,d)\in C\oplus D: \pi_C(c)=\pi_D(d)\}$. Take $\rho_C:P\rightarrow C$ and $\rho_D:P\rightarrow D$ be the restriction maps to the summands. Then we have the following \textit{pullback diagram} of $C^{\ast}$-algebras:
$$\xymatrix{
	P \ar[r]^{\rho_C}\ar[d]_{\rho_D} & C \ar[d]^{\pi_C} \\
	D \ar[r]_{\pi_D}                   & Q,
}$$

\begin{lemma}\cite[Proposition 2.7.15]{WillettYu-Book}\label{Kpullback}
	For a pullback diagram of $C^{\ast}$-algebras as above, there is a Mayer-Vietoris six-term exact sequence
	$$\xymatrix{
		K_0(P) \ar[r] & 
		K_0(C)\oplus K_0(D) \ar[r] & 
		K_0(Q)\ar[d] \\
		K_1(Q)\ar[u] &
		K_1(C)\oplus K_1(D)\ar[l] &
		K_1(P),\ar[l] 
	}$$
	which is natural for commutative maps between pullback diagram of $C^{\ast}$-algebras.
\end{lemma}

\begin{definition}\cite[Definition 2.7.11]{WillettYu-Book}\label{Def-quasi-stable}
	A $C^{\ast}$-algebra $B$ is called to be \textit{quasi-stable} if for all positive integer $n$ there exists an isometry $v$ in the multiplier algebra of $M_n(A)$ such that $vv^{\ast}$ is the matrix unit $e_{11}$.
\end{definition}

For quasi-stable $C^{\ast}$-algebras, we have the following advantage for the computation of the $K$-theory.

\begin{lemma}\cite[Proposition 2.7.12]{WillettYu-Book}\label{K-prod}
	Let $(B_i)_{i\in \mathcal{I}}$ be a family of quasi-stable $C^{\ast}$-algebras. Then the natural quotient maps
	$$\pi_j: \prod_{i\in \mathcal{I}}B_i \rightarrow B_j$$
	induce an isomorphism
	$$\prod_{i\in \mathcal{I}} \pi_{i,\ast}: K_{\ast}(\prod_{i\in \mathcal{I}}B_i) \rightarrow \prod_{i\in \mathcal{I}} K_{\ast}(B_i).$$
\end{lemma} 

Now, we begin to state the main result of this subsection.
\begin{proposition}\label{directproduct}
	Let $(X_i)_{i\in \N}$ be a sequence of metric spaces, $X=\sqcup X_i$ be the separated coarse union and $A$ be a $C^{\ast}$-algebra. Then the inclusion map $\tau$ from $C^{\ast}_L(X, A)$ to the uniform product $\prod_{i\in \N}C^{\ast}_L(X_i, A)$ induces an isomorphism 
	$$\tau_{\ast}:K_{\ast}(C^{\ast}_L(X, A)) \rightarrow K_{\ast}(\prod_{i\in \N}C^{\ast}_L(X_i, A)).$$
\end{proposition}

By the similar proof of Lemma \ref{Roequasi-stable}, we can show that localization algebra $C^{\ast}_L(X_i, A)$ is quasi-stable for any $C^{\ast}$-algebra $A$. Thus by Lemma \ref{K-prod}, we have the following corollary.

\begin{corollary}\label{K-homology-directproduct}
	Let $X_i$, $X$ and $A$ be as above, then the composition $(\prod_{i\in \N}\pi_{i,\ast})\circ\tau_{\ast}$ from $K_{\ast}(C^{\ast}_L(X, A))$ to $\prod_{i\in \N} K_{\ast}(C^{\ast}_L(X_i, A))$ is an isomorphism.
\end{corollary}
\begin{remark}
	The localization algebra was introduced by Yu in \cite{Yu-Localizationalg} whose $K$-theory provide a model for $K$-homology (see also \cite{QiaoRoe}\cite{DadarlatWillettWu}). Thus the above corollary actually implies the cluster axiom of $K$-homology.
\end{remark}

Before proving Proposition \ref{directproduct}, we introduce a uniformly equicontinuous product 
$$\prod^{ec}_{i\in \N}C^{\ast}_L(X_i, A)=\{(u_i)_{i\in \N} \in \prod_{i\in \N}C^{\ast}_L(X_i, A): \:(u_i)_{i\in \N} \:\:\mbox{is uniformly equicontinuous}\}.$$
And we fix the following notations. Let $A$ be a $C^{\ast}$-algebra acting on a Hilbert space $H_A$. Let $H_{X_i}=\ell^2(Z_{X_i})\otimes H_A\otimes H$, recall that $Z_{X_i}$ is a countable dense subset of $X_i$ and $H$ is a separable Hilbert space. And let $H^{\infty}_{X_i}=\ell^2(Z_{X_i})\otimes H_A\otimes (\oplus^{\infty}_{n=0}H)$. Choose a unitary $U':\oplus^{\infty}_{n=0}H \rightarrow H$ and define
$$U_i=I\otimes U': H^{\infty}_{X_i} \rightarrow H_{X_i}.$$
Define two isometries 
$$V_i: H_{X_i} \rightarrow H^{\infty}_{X_i},\:\:\:\: h \mapsto (h, 0, \cdots),$$
$$W_i: H^{\infty}_{X_i} \rightarrow H^{\infty}_{X_i},\:\:\:\: (h_0, h_1, \cdots) \mapsto (0, h_0, h_1, \cdots).$$

Then the proof of Proposition \ref{directproduct} is divided into the following two lemmas that are inspired by the proofs of Proposition 6.6.2 in \cite{WillettYu-Book} and Theorem 3.4 in \cite{WillettYu-ControlledKK}, respectively.

\begin{lemma}\label{locprodlemma1}
	The inclusion map $\tau: C^{\ast}_L(X, A)\rightarrow \prod^{ec}_{i\in \N}C^{\ast}_L(X_i, A)$ induces an isomorphism on the $K$-theory.
\end{lemma}
\begin{proof}
	Let $\mathfrak{I}$ be the algebra consisting of all elements $(u_i)_{i\in \N}\in \prod^{ec}_{i\in \N}C^{\ast}_{L}(X_i, A)$ which satisfies that $\lim_{t\rightarrow \infty} u_i(t)=0$ for every $i\in \N$. Then $\mathfrak{I}$ is an ideal of $\prod^{ec}_{i\in \N}C^{\ast}_{L}(X_i, A)$. We have the following two claims.
	
	\textbf{Claim 1}: $K_{\ast}(\mathfrak{I})=0$.
	
	Now, we prove the above claim. For every $i\in \N$, we define three homomorphisms $\alpha_i$ and $\beta_i$ on $C^{\ast}_L(X_i, A)$ by
	$$\alpha_i(u_i)(t)=U_i(u_i(t) \oplus 0 \oplus \cdots \oplus 0 \oplus \cdots)U^{\ast}_i,$$
	$$\beta_i(u_i)(t)=U_i(0\oplus u_i(t+1) \oplus \cdots \oplus u_i(t+n) \oplus \cdots)U^{\ast}_i,$$
	and 
	$$\gamma_i(u_i)(t)=\alpha_i(u_i)(t+1)+\beta_i(u_i)(t+1).$$
	Then we have that
	$$\alpha_i(u_i)=(U_iV_i)u_i(U_iV_i)^{\ast}\:\:\text{and}\:\:\beta(p_i)=(U_iW_iU^{\ast}_i)(\gamma(u_i))(U_iW^{\ast}_iU^{\ast}_i).$$
	Let $\alpha=\prod \alpha_i$, $\beta=\prod \beta_i$ and $\gamma=\prod \gamma_i$, then $\alpha_{\ast}$ is an isomorphism and $\beta_{\ast}=\gamma_{\ast}$ on $K_{\ast}(\mathfrak{I})$ by Lemma \ref{isometryequivalent}. On the other hand, for any $s\in [0,1]$, let $\gamma^{(s)}_i(u_i)(t)=\alpha_i(u_i)(t+s)+\beta_i(u_i)(t+s)$. Then $\prod \gamma^{(s)}_i$ is a continuous path connecting $\alpha+\beta$ and $\gamma$ by the uniform equicontinuity of $(u_i)_{i\in \N}$. Thus $\gamma_{\ast}=(\alpha+\beta)_{\ast}=\alpha_{\ast}+\beta_{\ast}$ on $K_{\ast}(\mathfrak{I})$. In conclusion, $\alpha_{\ast}=0$ which implies that $K_{\ast}(\mathfrak{I})=0$. Similarly, we can show that $K_{\ast}(C^{\ast}_L(X, A) \cap \mathfrak{I})=0$.
	
	\textbf{Claim 2}: the inclusion map $\tau$ from $\frac{C^{\ast}_L (X, A)}{C^{\ast}_L(X, A) \cap \mathfrak{I}}$ to $\frac{\prod^{ec}_{i\in \N}C^{\ast}_L(X_i, A)}{\mathfrak{I}}$ is an isomorphism.
	
	In order to prove this claim, we just need to show that $i$ is surjective. Choosing a sequence of increasing finite subsets $\mathcal{I}_n$ of $\N$ such that $\cup^{\infty}_{n=1} \mathcal{I}_n=\N$ and $\mathcal{I}_0=\emptyset$. For any $(u_i)_{i\in \N} \in \prod^{ec}_{i\in \N}C^{\ast}_{L}(X_i, A)$ and any positive integer $n$, there exists $t_n>0$ such that $\prop(u_i(t))<1/n$ for any $i\in \mathcal{I}_n$ and $t\geq t_n$. Assume $t_0=0$, $t_{n}+1<t_{n+1}$. For $i\in \mathcal{I}_n\setminus \mathcal{I}_{n-1}$, define 
	$$v_i(t_n)=\left \{ 
	\begin{array}{lcl}
		\sum_{k} \sqrt{\phi^{(i)}_{k,n}} u_i(t) \sqrt{\phi^{(i)}_{k,n}}, & \mbox{for} & t<t_{n-1}; \\
		\frac{t_n-t}{t_{n}-t_{n-1}}(\sum_{k} \sqrt{\phi^{(i)}_{k,n}} u_i(t) \sqrt{\phi^{(i)}_{k,n}})+\frac{t-t_{n-1}}{t_{n}-t_{n-1}}u_i(t), & \mbox{for} & t_{n-1}\leq t \leq t_n;\\
		u_i(t),  &  \mbox{for} & t>t_n,
		
	\end{array}
	\right.$$
	where $(\phi^{(i)}_{k,n})_{k}$ is a partition of unity subordinate to an open cover $\{B(x,1/n): \:x\in X_i\}$ of $X_i$. Then $\|v_i\|\leq 2\|u_i\|$ and $\sup_{i\in \N}\prop(v_i(t))<\frac{1}{n}$ when $t\geq t_{n+1}$ which implies $(v_i)_{i\in \N} \in C^{\ast}_{L}(X, A)$. Moreover, $\|v_i(t)-u_i(t)\|=0$ for all $t\geq t_n$ and $i\in \mathcal{I}_n\setminus \mathcal{I}_{n-1}$ which implies $(v_i)_{i\in \N}-(u_i)_{i\in \N}\in \mathfrak{I}$. Thus, $\tau$ is a surjective map.
	
	Combining the above two claims and the six-term exact sequence of $K$-theory, we completed the proof.
\end{proof}

\begin{lemma}\label{locprodlemma2}
	The inclusion map $\tau: \prod^{ec}_{i\in \N}C^{\ast}_L(X_i, A)\rightarrow \prod_{i\in \N}C^{\ast}_L(X_i, A)$ induces an isomorphism on the $K$-theory.
\end{lemma}
\begin{proof}
	We prove the lemma by the following two steps.
	
	\textbf{Step 1}: consider two subspaces $E=\sqcup^{\infty}_{n=0}[2n, 2n+1]$ and $F=\sqcup^{\infty}_{n=0}[2n+1, 2n+2]$ of $[0,\infty)$, then $E\cap F=\mathbb{N}^+$. Define $C^{\ast}_{L,E}(X_i, A)$ to be the norm closure of the $\ast$-algebra consisting of all bounded and uniformly continuous functions $u_i:E\rightarrow C^{\ast}(X_i, A)$ satisfying that 
	$$\sup_{t\in E}(\prop(u_i(t)))<\infty\:\:\text{and}\:\:\ \lim_{t\rightarrow \infty}(\prop(u_i(t)))=0.$$
	Similarly define $C^{\ast}_{L, F}(X_i, A)$ and $C^{\ast}_{L, \mathbb{N}^+}(X_i, A)$. 
	Then by Lemma \ref{Kpullback}, we have the following commutative diagram
	$$\tiny{\xymatrix{
			\:\ar[r]&
			K_1(\prod^{ec}_{i\in \N}C^{\ast}_{L, \mathbb{N}^+}(X_i, A)) \ar[r] \ar[d]^{\tau_1} & 
			K_0(\prod^{ec}_{i\in \N}C^{\ast}_{L}(X_i, A)) \ar[r] \ar[d]^{\tau_0} & 
			\oplus_{j=E, F} K_0(\prod^{ec}_{i\in \N}C^{\ast}_{L, j}(X_i, A))
			\ar[r] \ar[d]^{\tau_0\oplus \tau_0} &
			\: \\
			\:\ar[r]&
			K_1(\prod_{i\in \N}C^{\ast}_{L, \mathbb{N}^+}(X_i, A)) \ar[r] & 
			K_0(\prod_{i\in \N}C^{\ast}_{L}(X_i, A)) \ar[r] & 
			\oplus_{j=E, F} K_0(\prod_{i\in \N}C^{\ast}_{L, j}(X_i, A))
			\ar[r] &
			\:.
	}}$$
	Since we have $\prod^{ec}_{i\in \N}C^{\ast}_{L, \mathbb{N}^+}(X_i, A)=\prod_{i\in \N}C^{\ast}_{L, \mathbb{N}^+}(X_i, A)$, thus the proof is reduced to verifying $\tau_{\ast}: K_{\ast}(\prod^{ec}_{i\in\N}C^{\ast}_{L, J}(X_i, A)) \rightarrow K_{\ast}(\prod_{i\in \N}C^{\ast}_{L, J}(X_i, A))$ are isomorphic for $J=E, F$. 
	
	\textbf{Step 2}: We will show that $\tau_{\ast}: K_{\ast}(\prod^{ec}_{i\in \N}C^{\ast}_{L, J}(X_i, A)) \rightarrow K_{\ast}(\prod_{i\in \N}C^{\ast}_{L, J}(X_i, A))$ is an isomorphism for the case of $J=E$ and the case of $J=F$ is similar. Let $C^{\ast}_{L,E,0}(X_i, A)=\{u_i\in C^{\ast}_{L,E}(X_i, A): u_i(2n)=0, n\in \mathbb{N}\}$. For any element $(u_i)_{i\in \N}\in \prod^{ec}_{i\in \N}C^{\ast}_{L, E, 0}(X_i, A)$ and $s\in [0,1]$, let $u^{(s)}_i(t)=u_i(2n+s(t-2n)),\:\: t\in [2n,2n+1]$. Since that $(u_i)_{i\in \N}$ is uniformly equicontinuous, thus $(u^{(s)}_i)_{i\in \N}$ is a continuous path connecting $(u_i)_{i\in \N}$ and $0$ in $\prod^{ec}_{i\in \N}C^{\ast}_{L,E,0}(X_i, A)$ which implies that $K_{\ast}(\prod^{ec}_{i\in \N}C^{\ast}_{L,E,0}(X_i, A))=0$.
	
	Next we will prove that $K_{0}(\prod_{i\in \N}C^{\ast}_{L,E,0}(X_i, A))=0$. For any projection $(p_i)_{i\in \N} \in \prod_{i\in \N}C^{\ast}_{L, E, 0}(X_i, A)^{+}$, where $C^{\ast}_{L, E, 0}(X_i, A)^{+}$ is the unitization of $C^{\ast}_{L, E, 0}(X_i, A)$. Since that $p_i$ is uniformly continuous for every $i\in \N$, thus there exists a positive number $0<\delta_i<1$ such that $\|p_i(t)-p_i(t')\|<1$ for any $t,t'\in E$ with $|t-t'|\leq 1-\delta_i$. For every $m\in \mathbb{N}$, define 
	$$p^{(m)}_i(t)=p_i(2n+(t-2n)\delta^{m}_{i}),\:\:t\in [2n,2n+1].$$
	Then we define three projections $\alpha_{i}(p_i)$, $\beta_{i}(p_i)$ and $\gamma_{i}(p_i)$ in $C^{\ast}_{L, E, 0}(X_i, A)^{+}$ by
	$$\alpha_{i}(p_i)(t)=U_i(p_i(t) \oplus 0 \oplus \cdots \oplus 0 \oplus \cdots)U^{\ast}_i,$$ 
	$$\beta_{i}(p_i)(t)=U_i(0\oplus p^{(1)}_i(t) \oplus \cdots \oplus p^{(m)}_i(t) \oplus \cdots)U^{\ast}_i.$$
	and 
	$$\gamma_{i}(p_i)(t)=U_i(0\oplus p_i(t) \oplus \cdots \oplus p^{(m-1)}_i(t) \oplus \cdots)U^{\ast}_i.$$
	Then we have that 
	$$\alpha_i(p_i)=(U_iV_i)(p_i)(U_iV_i)^{\ast}\:\:\text{and}\:\: \gamma_{i}(p_i)=(U_iW_iU^{\ast}_i)((\alpha_i+\beta_i)(p_i))(U_iW_iU^{\ast}_i)^{\ast}.$$
	Then $[(p_i)_{i\in \N}]=[(\alpha_i(p_i))_{i\in \N}]$ and $[(\gamma_{i}(p_i))_{i\in \N}]=[(\alpha_i(p_i))_{i\in \N}]+[(\beta_i(p_i))_{i\in \N}]$ in $K_{0}(\prod_{i\in \N}C^{\ast}_{L,E,0}(X_i, A))$ by Lemma \ref{isometryequivalent}. On the other hand, since 
	$$\sup_{i\in \N}\|\gamma_{i}(p_i)-\beta_{i}(p_i)\|<1,$$
	thus $[(\gamma_{i}(p_i))_{i\in \N}]=[(\beta_{i}(p_i))_{i\in \N}]$ in $K_{0}(\prod_{i\in \N}C^{\ast}_{L,E,0}(X_i, A))$ by Lemma \ref{smallhomotopy}. In conclusion, we have that $[(p_i)_{i\in \N}]=0$ which implies that $K_{0}(\prod_{i\in \N}C^{\ast}_{L,E,0}(X_i, A))=0$. Similarly, we can prove that $K_{1}(\prod_{i\in \N}C^{\ast}_{L,E,0}(X_i, A))=0$ by the suspension discussion.
	
	Moreover, since that 
	$$\prod_{i\in \N}C^{\ast}_{L,E}(X_i, A) / \prod_{i\in \N}C^{\ast}_{L,E,0}(X_i, A)=\prod^{ec}_{i\in \N}C^{\ast}_{L,E}(X_i, A) / \prod^{ec}_{i\in \N}C^{\ast}_{L,E,0}(X_i, A),$$ 
	thus by the six-term exact sequence of $K$-theory, $\tau_{\ast}$ induces an isomorphism between $K_{\ast}(\prod^{ec}_{i\in \N}C^{\ast}_{L, E}(X_i, A))$ and $K_{\ast}(\prod_{i\in \N}C^{\ast}_{L, E}(X_i, A))$.    	   	  	
\end{proof}

Combining Lemma \ref{locprodlemma1} with Lemma \ref{locprodlemma2}, we completed the proof of Proposition \ref{directproduct}.

\subsection{Relations}

Let us first recall the coarse Baum-Connes conjecture with coefficients and its uniform version. 

Recall that $e_A$ is an evaluation at zero map from $C^{\ast}_L(X, A)$ to $C^{\ast}(X, A)$. The following conjecture is called to be the \textit{coarse Baum-Connes conjecture with coefficients}.

\begin{conjecture}\label{CBC}
	Let $X$ be a proper metric space, $N_{X}$ be a locally finite net in $X$ and $A$ be a $C^{\ast}$-algebra, then the following homomorphism
	$$e_{A, \ast}: \lim_{k\rightarrow \infty}K_{\ast}(C_L^{\ast}(P_{k}(N_X), A)) \rightarrow \lim_{k\rightarrow \infty}K_{\ast}(C^{\ast}(P_{k}(N_X), A))$$
	is an isomorphism between abelian groups, where $P_k(N_X)$ is the Rips complex. 
\end{conjecture}

Let $X$ be a proper metric space and $\sqcup_{\N} X$ be the separated coarse union of a sequence of metric spaces $X$ (Recall from Subsection \ref{subsection-K-homology-direct-product}).

\begin{definition}\label{Def-uniform-CBC}
	Call $X$ satisfies the \textit{uniform coarse Baum-Connes conjecture with coefficients} in $A$, if the coarse Baum-Connes conjecture with coefficients in $A$ holds for $\sqcup_{\N} X$.
\end{definition}

In \cite{Zhang-CBCFC}, we introduced the coarse Baum-Connes conjecture with filtered coefficients and proved the following lemma (cf. \cite[Corollary 4.12]{Zhang-CBCFC}).

\begin{lemma}\label{Lem-CBCC-UCBC}
	Let $X$ be a proper metric space and $A$ be a $C^{\ast}$-algebra. If the coarse Baum-Connes conjecture with coefficients in $\ell^{\infty}(\N, A\otimes \K(H))$ holds for $X$, then the uniform coarse Baum-Connes conjecture with coefficients in $A$ holds for $X$.
\end{lemma}

Now we state and prove the relations between the coarse Baum-Connes conjecture with coefficients and its quantitative version.

\begin{theorem}\label{quan-CBC-CBC}
	The quantitative coarse Baum-Connes conjecture with coefficients in $A$ implies the coarse Baum-Connes conjecture with coefficients in $A$.
\end{theorem}   
\begin{proof}
	Assume a proper metric space $X$ satisfies the quantitative coarse Baum-Connes conjecture with coefficients in $A$. Next we will prove that the homomorphism $e_{A, \ast}$ appeared in the coarse Baum-Connes conjecture with coefficients in $A$ is an isomorphism.
	
	Firstly, for $x\in K_{\ast}(C^{\ast}_L(P_k, A))$, if $e_{A, \ast}(x)=0$ in $\lim_{l\rightarrow \infty} K_{\ast}(C^{\ast}(P_{l}, A))$. Then there exists $k'\geq k$ such that $ad_{i_{kk'},\ast}(e_{A, \ast}(x))=0$ in $K_{\ast}(C^{\ast}(P_{k'}, A))$. Thus, $e_{A, \ast}(Ad_{i_{kk'},\ast}(x))=0$ which implies that $\iota^{\ep/16, r}_{\ast}\circ \mu^{\ep/16, r}_{k',A,\ast}(Ad_{i_{kk'},\ast}(x))=0$ for any $0<\ep<1/4$ and $r>0$ by Lemma \ref{quan-CBC-and-CBC}. And by Lemma \ref{quan-K-and-K}, there exists $r'\geq r$ such that $\iota^{\ep/16, \ep, r, r'}_{\ast}\circ \mu^{\ep/16, r}_{k',A, \ast}(Ad_{i_{kk'},\ast}(x))=0$ in $K^{\ep, r'}_{\ast}(C^{\ast}(P_{k'}, A))$. Then by Lemma \ref{quan-CBC-and-CBC} again, we have $\mu^{\ep, r'}_{k',A, \ast}(Ad_{i_{kk'}, \ast}(x))=0$ for any $0<\ep<1/4$. Thus by the first part (\ref{quan-CBC1}) of Conjecture \ref{quan-CBC}, there exists $k''\geq k'$ such that $Ad_{i_{kk''}, \ast}(x)=0$ which implies that $e_{A, \ast}$ is injective. \par
	Secondly, for any $y\in K_{\ast}(C^{\ast}(P_{k}, A))$. By Lemma \ref{quan-K-and-K}, for any $0<\ep<1/4$ there exists $r>0$ and $y'\in K^{\ep, r}_{\ast}(C^{\ast}(P_{k}, A))$ such that $\iota^{\ep, r}(y')=y$. Then by the second part (\ref{quan-CBC2}) of Conjecture \ref{quan-CBC}, there exist $k'\geq k$, $r'\geq r+1$, $1/4>\ep'\geq \ep$ and $x\in K_{\ast}(C^{\ast}_L(P_{k'}, A))$ such that $\mu^{\ep', r'}_{k',A, \ast}(x)=\iota^{\ep, \ep', r_{k'}, r'}_{\ast}\circ ad^{\ep, r, r_{k'}}_{i_{kk'}, \ast}(y')$, where $r_{k'}=r+\frac{1}{2^{k'}}$. Thus by Lemma \ref{quan-CBC-and-CBC}, we have that $e_{A, \ast}(x)=\iota^{\ep, r_{k'}}_{\ast}\circ ad^{\ep, r, r_{k'}}_{i_{kk'}, \ast}(y')=ad_{i_{kk'},\ast}(\iota^{\ep, r_{k'}}_{\ast}(y'))=ad_{i_{kk'},\ast}(y)$ which implies that $e_{A, \ast}$ is surjective.
\end{proof}

For the reverse direction, we have the following theorem.

\begin{theorem}\label{uniformCBC=quanCBC}
	Let $X$ be a proper metric space and $A$ be a $C^{\ast}$-algebra. Then the followings are equivalent:
	\begin{enumerate}
		\item\label{uniformCBC=quanCBC1} $X$ satisfies the uniform coarse Baum-Connes conjecture with coefficients in $A$;
		\item\label{uniformCBC=quanCBC2} $X$ satisfies the quantitative coarse Baum-Connes conjecture with coefficients in $A$. 
	\end{enumerate}
\end{theorem}
\begin{proof}
	Let $N_X$ be a locally finite net in $X$. The main ingredient of the proof is the following commutative diagram which is obtained by Corollary \ref{K-homology-directproduct} and Corollary \ref{uniform-product-Roealg}:
	$$\xymatrix{
		K_{\ast}(C^{\ast}_{L}(\sqcup_{\N} P_{k}(N_{X}), A)) \ar[d]_{\cong} \ar[r]^{\mu^{\ep, r}_{k, A, \ast}} & \:\:K^{\ep, r}_{\ast}(C^{\ast}(\sqcup_{\N} P_{k}(N_{X}), A)) \ar[d]^{\cong} \\
		\prod_{\N} K_{\ast}(C^{\ast}_{L}(P_{k}(N_{X}), A)) \ar[r]^{\:\:\prod_{\N}\mu^{\ep, r}_{k, A, \ast}\:\:} & \:\:\:\:\prod_{\N} K^{\ep, r}_{\ast}(C^{\ast}(P_{k}(N_{X}), A)).
	}$$
	
	Firstly, we prove (\ref{uniformCBC=quanCBC2}) $\Rightarrow$ (\ref{uniformCBC=quanCBC1}). For the injection, if an element $x\in K_{\ast}(C^{\ast}_{L}(\sqcup_{\N} P_{k}(N_X), A))$ with $e_{A, \ast}(x)=0$ in $K_{\ast}(C^{\ast}(\sqcup_{\N} P_{k}(N_X), A))$. Since we assume (\ref{uniformCBC=quanCBC2}) holds, thus there exists $16\ep<1/4$ corresponding to $k$ in the first part of Conjecture \ref{quan-CBC}. Then by the first diagram in Lemma \ref{quan-CBC-and-CBC},  we have that 
	$$\iota^{\ep, r}\circ \mu^{\ep, r}_{k, A, \ast}(x)=e_{A, \ast}(x)=0$$
	for any $r>0$. Fixed a positive number $r$, by Lemma \ref{quan-K-and-K}, there exists $r'\geq r$ such that $\iota^{\ep, 16\ep, r, r'}_{\ast}\circ \mu^{\ep, r}_{k, A, \ast}(x)=0$. Then by the second commutative diagram in Lemma \ref{quan-CBC-and-CBC}, we obtain that
	$$\mu^{16\ep, r'}_{k, A, \ast}(x)=\iota^{\ep, 16\ep, r, r'}_{\ast}\circ \mu^{\ep, r}_{k, A, \ast}(x)=0.$$
	Let $(x_i)_{i\in \N}\in \prod_{\N} K_{\ast}(C^{\ast}_{L}(P_{k}(N_{X}), A))$ corresponding to $x$ under the left-hand side vertical isomorphism of the above commutative diagram, then $\mu^{16\ep, r'}_{k, A, \ast}(x_i)=0$ in $K^{16\ep, r'}_{\ast}(C^{\ast}(P_{k}(N_X), A))$ for any $i\in \N$ by the above commutative diagram. Because (\ref{uniformCBC=quanCBC2}) holds, thus by the first part of Conjecture \ref{quan-CBC}, there exists $k'\geq k$ such that $Ad_{i_{kk'}, \ast}(x_i)=0$ for all $i\in \N$. Thus, $Ad_{i_{kk'}, \ast}(x)=0$ in $K_{\ast}(C^{\ast}_{L}(\sqcup_{\N} P_{k'}(N_X), A))$ by the naturality of the left-hand side vertical isomorphism of the above diagram.
	
	And we can similarly prove it for the case of surjection.
	
	Secondly, we prove (\ref{uniformCBC=quanCBC1}) $\Rightarrow$ (\ref{uniformCBC=quanCBC2}). Assume the first part of Conjecture \ref{quan-CBC} is not true, then there exists $k\in N$ satisfying that for any $0<\ep<1/4$, there exists $r>0$ such that for any $k'\geq k$, the quantitative statement $QI_{A}(k,k', \ep, r)$ does not holds, namely, there exists $x_n\in K_{\ast}(C^{\ast}_{L}(P_{k}(N_X), A))$ with $\mu^{\ep, r}_{k, A, \ast}(x_n)=0$ but $Ad_{i_{kk_n}, \ast}(x_n)\neq 0$ in $K_{\ast}(C^{\ast}_{L}(P_{k_n}(N_X), A))$ for any $k_n=k+n$, $n=0, 1,2, \cdots$. Let $x\in K_{\ast}(C^{\ast}_{L}(\sqcup_{\N}P_{k}(N_X), A))$ corresponding to $(x_n)_{i\in \N}$ under the left-hand side vertical isomorphism in the above commutative diagram, then $\mu^{\ep, r}_{k, A, \ast}(x)=0$ in $K^{\ep, r}_{\ast}(C^{\ast}(\sqcup_{\N} P_{k}(N_X), A))$ by the above commutative diagram. Thus, by the first commutative diagram in Lemma \ref{quan-CBC-and-CBC}, we have that
	$$e_{A, \ast}(x)=\iota^{\ep, r}_{\ast}\circ \mu^{\ep, r}_{k, A, \ast}(x)=0.$$
	By the assumption (\ref{uniformCBC=quanCBC1}), there exists $k_m=k+m\geq k$ such that $Ad_{i_{kk_m},\ast}(x)=0$. Then by the naturality of the left-hand side vertical isomorphism of the above diagram again, we have that $Ad_{i_{kk_m},\ast}(x_m)=0$ which contradicts with the assumption.
	
	Similarly, we can prove that the second part of Conjecture \ref{quan-CBC} is true providing the statement (\ref{uniformCBC=quanCBC1}) holds.
\end{proof}

  Combining the above theorem with Lemma \ref{Lem-CBCC-UCBC}, we obtain the following corollary.
  
\begin{corollary}\label{Cor-CBCFC-QCBC}
	Let $X$ be a proper metric space and $A$ be a $C^{\ast}$-algebra. If $X$ satisfies the coarse Baum-Connes conjecture with coefficients in $\ell^{\infty}(\N, A\otimes \K(H))$, then $X$ satisfies the quantitative coarse Baum-Connes conjecture with coefficients in $A$.
\end{corollary}

\begin{remark}
	The methods to the coarse Baum-Connes conjecture also work for the coarse Baum-Connes conjecture with coefficients (cf. \cite{Zhang-CBCFC}), thus there are many examples satisfying the quantitative coarse Baum-Connes conjecture with coefficients by the above corollary.
\end{remark}

\section{Reducing the conjecture to a sequence of bounded metric spaces}\label{Sec-Reduction}

Inspired by \cite[Section 12.5]{WillettYu-Book}, in this section, we will reduce the quantitative coarse Baum-Connes conjecture with coefficients for a proper metric space to a uniform version of the conjecture for a sequence of bounded subspaces. Firstly, we introduce the \textit{uniformly quantitative coarse Baum-Connes conjecture with coefficients} for a family of metric spaces.

\begin{conjecture}\label{uniform-quan-CBC}
	Let $(X_i)_{i\in \mathcal{I}}$ be a family of proper metric spaces. Then the followings are true.
	\begin{enumerate}
		\item\label{uniform-quan-CBC1} For any $k\in \N$, there exists $0<\ep<1/4$ satisfying that for any $r>0$, there exists $k'\geq k$ such that $QI_A(k, k', \ep, r)$ hold for all $X_i$.
		\item\label{uniform-quan-CBC2} For any $k\in \N$, there exists $0<\ep<1/4$ satisfying that for any $r>0$, there exist $k'\geq k$, $r'\geq r+1$ and $1/4>\ep'\geq \ep$ such that $QS_A(k, k', \ep, \ep', r, r')$ hold for all $X_i$.
	\end{enumerate} 
\end{conjecture}
Secondly, we consider the quantitative coarse Baum-Connes conjecture with coefficients for coarse unions by using the above conjecture.

\begin{definition}\label{Def-coarse-union}
	A \textit{coarse union} of a sequence of metric spaces $(X_i, d_i)_{i\in \N}$, denoted by $\square X_i$, is defined to be the disjoint union $\sqcup_{\N} X_i$ equipped with a metric $d$ satisfying the following conditions:
	\begin{itemize}
		\item $d$ is finite-valued;
		\item $d(x,x')=d_{i}(x,x')$ for all $x,x'\in X_i$;
		\item $\lim_{n\rightarrow \infty} d(X_n, (\sqcup_{\N} X_i)\setminus X_n)=\infty$.
	\end{itemize}
\end{definition}

\begin{lemma}\label{quan-CBC-union}
	If the uniformly quantitative coarse Baum-Connes conjecture with coefficients in $A$ holds for $(X_i)_{i\in \N}$, then the quantitative coarse Baum-Connes conjecture with coefficients in $A$ holds for the coarse union $\square X_i$.
\end{lemma}
\begin{proof}
	By the assumption and Lemma \ref{equa-quan-CBC}, for any $k\in \N$, there exists $0<\ep_k<1/(4\max\{\lambda, \lambda_{\mathcal{D}}\})$ satisfying that for any $r>0$, there exist $k'\geq k$, $1/(4\max\{\lambda, \lambda_{\mathcal{D}}\})>\ep'\geq \ep_k$ and $r'\geq r+1$ such that
	$$\iota^{\ep_k, \ep',r_{k'},r'}_{\ast} \circ Ad^{\ep_k, r, r_{k'}}_{i_{kk'}, \ast}: K^{\ep_k, r}_{\ast}(C^{\ast}_{L, 0}(P_k(N_{X_i}), A))\rightarrow K^{\ep', r'}_{\ast}(C^{\ast}_{L, 0}(P_{k'}(N_{X_i}), A))$$
	is a zero map for all $i\in \N$, where $N_{X_i}$ is a locally finite net in $X_i$, $r_{k'}=r+1/2^{k'}$ and $\lambda, \lambda_{\mathcal{D}}\geq 1$ appeared in Lemma \ref{six-term-seq}. Since the quantitative coarse Baum-Connes conjecture with coefficients is a coarsely equivalent invariant (see Lemma \ref{Lem-coarse-equi-QCBC}), thus we can assume that 
	$$d(X_n, X_{n'})> \text{diam}(X_0\cup X_1 \cup \cdots \cup X_n)+1 \:\:\text{for}\:\: n'>n.$$
	Then for any $k\in \N$ and $r>0$, let $\ep_k$, $\ep'$, $r'$ be as above, choose $k'''\geq k''\geq k'$ such that $d(X_{k''-1}, X_{k''})>r'$ and $\text{diam}(X_0\cup X_1 \cup \cdots \cup X_{k''-1}) \leq k'''\leq d(X_{k''-1}, X_{k''})$, thus by Corollary \ref{uniform-product-Roealg} and Corollary \ref{K-homology-directproduct}, we have that
	\begin{footnotesize}
		\begin{equation*}
			K^{\ep_k, r}_{\ast}(C^{\ast}_{L, 0}(P_k(N_{\square X_i}), A))=K^{\ep_k, r}_{\ast}(C^{\ast}_{L, 0}(P_{k}(\sqcup_{i=0}^{k''-1} N_{X_i}))) \oplus \prod_{i\geq k''} K^{\ep_k, r}_{\ast}(C^{\ast}_{L,0}(P_{k}(N_{X_i}), A)),
		\end{equation*}
	\end{footnotesize}
	and
	\begin{footnotesize}
		\begin{equation*}
			K^{\ep', r'}_{\ast}(C^{\ast}_{L, 0}(P_{k'''}(N_{\square X_i}), A))=K^{\ep', r'}_{\ast}(C^{\ast}_{L, 0}(P_{k'''}(\sqcup_{i=0}^{k''-1} N_{X_i}))) \oplus \prod_{i\geq k''} K^{\ep', r'}_{\ast}(C^{\ast}_{L,0}(P_{k'''}(N_{X_i}), A)).
		\end{equation*}
	\end{footnotesize}
	Let $r''=\max\{r', \text{diam}(P_{k'''}(\sqcup_{i=0}^{k''-1} N_{X_i}))\}$, then for any $k\in \N$ and $r>0$, we have that
	$$\iota^{\ep_k, \ep',r_{k'},r'}_{\ast} \circ Ad^{\ep_k, r, r_{k'}}_{i_{kk'}, \ast}: K^{\ep_k, r}_{\ast}(C^{\ast}_{L, 0}(P_k(N_{\square X_i}), A))\rightarrow K^{\ep', r''}_{\ast}(C^{\ast}_{L, 0}(P_{k'''}(N_{\square X_i}), A))$$
	is a zero map, which completed the proof by Lemma \ref{equa-quan-CBC} again. 
\end{proof}

For a metric space $X$, fixed a point $x_0\in X$, let 
$$X_n=\{x\in X: n^3-n\leq d(x, x_0) \leq (n+1)^3+(n+1)\},$$
which is a bounded subspace of $X$ for every $n\geq 0$. Let 
$Y=\sqcup_{n\: even} X_n$ and $Z=\sqcup_{n\: odd} X_n$ which are two subspaces of $X$, then $X=Y\cup Z$ and $Y$, $Z$, $Y\cap Z$ are three coarse unions of $(X_i)_{i\: even}$, $(X_j)_{j\: odd}$ and $(X_n\cap Z)_{n\: even}$, respectively. For any $s>0$, let 
$$\triangle_{even}=\{u\in C^{\ast}_{L, 0}(X): (x_t, y_t)\in\supp(u(t))\:\text{with}\: x_t\in Y,\: d(x_t, y_t)\leq s\},$$
similar to define $\triangle_{odd}$, then $(\triangle_{even}, \triangle_{odd}, C^{\ast}_{L, 0}(Y), C^{\ast}_{L, 0}(Z))$ is an $s$-controlled weak Mayer-Vietoris pair with coercitivity $1$ and $C^{\ast}_{L, 0}(Y) \cap C^{\ast}_{L, 0}(Z)=C^{\ast}_{L, 0}(Y\cap Z)$. Thus, by Lemma \ref{Control-MV} and Lemma \ref{equa-quan-CBC}, we have the following lemma. 

\begin{lemma}\label{Reduction-coarse-union}
	Let $X$, $Y$ and $Z$ be as above, if the coarse unions $Y$, $Z$ and $Y\cap Z$ satisfy the quantitative coarse Baum-Connes conjecture with coefficients in $A$, then $X$ also satisfies the quantitative coarse Baum-Connes conjecture with coefficients in $A$.
\end{lemma}

Combining Lemma \ref{quan-CBC-union} and Lemma \ref{Reduction-coarse-union}, we obtain the following theorem.

\begin{theorem}\label{redution-bounded-case}
	Let $X$, $X_n$ and $Z$ be as above. If the uniformly quantitative coarse Baum-Connes conjecture with coefficients in $A$ holds for $(X_n)_{n\:even}$, $(X_n)_{n\:odd}$ and $(X_n\cap Z)_{n\:even}$, then the quantitative coarse Baum-Connes conjecture with coefficients in $A$ holds for $X$.
\end{theorem}

By the above theorem, the quantitative coarse Baum-Connes conjecture with coefficients can be reduced to the uniformly quantitative coarse Baum-Connes conjecture with coefficients for a sequence of bounded subspaces.

\bibliographystyle{plain}
\bibliography{quan-CBC}

\end{document}